\documentclass{amsart}

\usepackage{
	algorithm,algpseudocode,
	amsmath,amsthm,amssymb,amsfonts,amscd,amsxtra,
	caption,color,courier,
	gensymb,hhline,kotex,lipsum,
	mathrsfs,mathtools,mdframed,multirow,
	soul,subcaption,stackrel,textcomp,tikz,
	verbatim,xcolor}
\usepackage[foot]{amsaddr}

\mathtoolsset{showonlyrefs}

\theoremstyle{plain}
\newtheorem{thm}{Theorem}

\newtheorem*{thm*}{thm}
\newtheorem*{lem*}{Lemma}
\newtheorem*{prop*}{Proposition}
\newtheorem*{cor*}{Corollary}
\newtheorem*{conj*}{Conjecture}

\theoremstyle{definition}
\newtheorem{defn}[thm]{Definition}

\title{The boundary of Rauzy fractal and discrete tilings}
\author[W. Choi]{Woojin Choi$^1$}
\author[H. Kang]{Hyosang Kang$^1$}
\author[J. Lee]{Jeonghoon Lee$^2$}
\author[Y. Oh]{Youchan Oh$^3$}

\address{$^1$Daegu Gyeongbuk Institute of Science and Technology (DGIST), Daegu 42988, South Korea.}
\address{$^2$Gyeonggi Science High School (GSHS), Gyeonggi-do 16297, South Korea.}
\address{$^3$Seoul Science High School (SSHS), Seoul 03066, South Korea.}
\address{$^\ast$ Corresponding author}
\email{hyosang@dgist.ac.kr}

\begin{document}

\maketitle 
\begin{abstract}
The Rauzy fractal is a domain in the two-dimensional plane constructed 
by the Rauzy substitution, a substitution rule on three letters.
The Rauzy fractal has a fractal-like boundary, 
and the currently known its constructions is not only for its boundary 
but also for the entire domain. 
In this paper,
we show that all points in the Rauzy fractal have a layered structure.
We propose two methods of constructing the Rauzy fractal using layered structures.
We show how such layered structures can be used to construct 
the boundary of the Rauzy fractal with less computation than conventional methods.
There is a self-replicating pattern in one of the layered structure in the Rauzy fractal.
We introduce a notion of self-replicating word
and visualize how some self-replicating words on three letters
creates discrete tiling of the two dimensional plane.
\end{abstract}


\
\section{Introduction}

In \cite{Ra82}, Rauzy proposed a method of constructing 
a compact region called the \textbf{Rauzy fractal} (Figure \ref{fig:rauzy-three-colors}).
There are two characteristics in the Rauzy fractal.
One is that the Rauzy fractal has a fractal-like boundary (as its name indicates),
and another is that it discretely tiles the two-dimensional plane (Figure \ref{fig:rauzy-tiling}).
The tiling characteristic of the Rauzy fractal generalizes to the Pisot conjecture,
which states that every Pisot substitution ($\S\ref{sec:rauzy-fractal}$) 
on $d$ letters gives a discrete tiling of the $d$-dimensional space $\mathbb R^d$ \cite{AI00, Ba18, BS07, RA04, So97}.
Pisot substitutions are also studied in relationship with Pisot numbers \cite{Pi38,To20}.
In symbolic dynamics, Pisot substitution plays one of key role 
in understanding the substitutive systems \cite{Ak15, Qu10}. 
Pisot conjecture further generalizes to the pure discrete spectrum conjecture
which states that every dynamical system defined by a Pisot substitution on $d$ letters 
has a pure discrete spectrum, which has been proved only for $d\le 2$ \cite{BM02,BK06,Si04,SS02}.

\begin{figure}
\centering
\begin{subfigure}[t]{0.45\textwidth}
\centering
    \includegraphics[scale=.3]{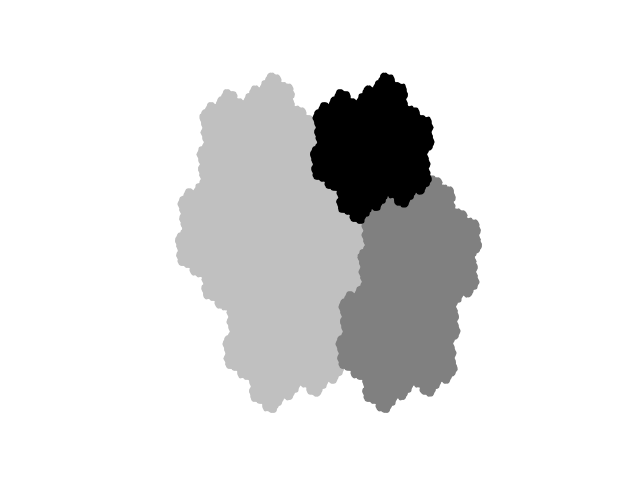}
    \caption{Rauzy fractal in three colors.}
    \label{fig:rauzy-three-colors}
\end{subfigure}
\begin{subfigure}[t]{0.45\textwidth}
\centering
	\includegraphics[scale=.27]{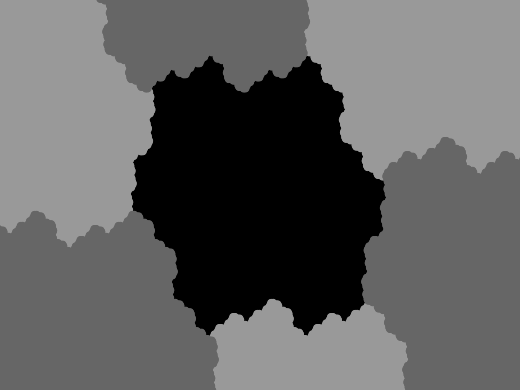}
	\caption{Tiling of the Rauzy fractal}
	\label{fig:rauzy-tiling}
\end{subfigure}
\caption{The Rauzy fractal and its tiling}
\label{fig:rauzy}
\end{figure}

There are two ways known to construct Rauzy fractal. 
One uses a convergent sequence of $3$-dimensional points \cite{Ra82},
and the other uses the exductive method \cite{AI00}.
Both methods construct the entire Rauzy fractal,
and increase the resolution of the Rauzy fractal with larger computations
for its boundary and interior at the same time.
Since the Rauzy fractal is a simply connected domain, 
so the higher resolution is only visible on its boundary. 
So it is questionable whether we can effectively increase the resolution of the boundary
of the Rauzy fractal only.

Our main goal is to find an efficient algorithm for drawing the boundary of the Rauzy fractal.
We found two new construction methods for the Rauzy fractal.
Both methods rely on the fact that Rauzy fractal is the union of ``layered" points.
The first method, the construction A ($\S\ref{sec:rauzy-const-1}$)
uses Theorem \ref{thm:main-1}, which states that the union of all A-layers 
(Equations \eqref{eqn:vi-layer-1} - \eqref{eqn:vi-layer-5}) is the Rauzy fractal.
To define the points in the A-layers, we use the tribonacci words (c.f. Equation \eqref{eqn:an-recursive}).
Each A-layer has its \textit{level}, and the points in the A-layers at the same level
cluster into hexagons, called \textit{cells}.
Thus we can visualize the boundary of the Rauzy fractal with
the boundary cells (c.f. Figures \ref{fig:rauzy-boundary}).

The second method, the construction B ($\S\ref{sec_alternative_construction}$)
uses Theorem \ref{thm:main-2}, which states that the union of all B-layers is the Rauzy fractal.
The B-layers are defined by a pseudo-\textit{self-replicating} pattern 
in the sequence of $0102010$ (Equations \eqref{eq-sn-3} - \eqref{eq-sn-9}).
The points in the B-layers are determined only by this pattern, 
whereas the points in the A-layers requires extra information on tribonacci words.
The word ``pseudo" comes from the exception at the number $2$ in the sequence $0102010$,
which we must have in order to create the exact Rauzy fractal.
Interestingly, we still get compact domains when we use fullly self-replicating patterns of various sequence.
In $\S\ref{sec_discrete_tiling_plane}$, we will visualize certain patterns that 
create compact domains which tiles $\mathbb R^2$ discretely (Figures \ref{fig-w-self}, \ref{fig:0120_tiling}).

\section{Pisot substitutions and their domains} \label{sec:rauzy-fractal}

Let the \textbf{letters} be the integers $0, 1, 2, \cdots$,
and a \textbf{word} be a string of letters. 
Given a sequence $\mathbf w = (w_0,\cdots, w_l)$ of integers, 
we denote a word $[\mathbf w]$ as
\begin{equation}\label{eqn:w-word}
	[\mathbf w] = [w_0w_1\cdots w_l].
\end{equation}
We call each $w_i$ a \textbf{character}.
A \textbf{word on $d$ letters} is a word whose characters are in the $d$ letters from $0$ to $d-1$.
The \textbf{empty word} is the word without any character, and denoted by $[\,]$. 
The \textbf{length} of the word is the number of its characters.
For example, $L([\,]) =0$ and $L([w_0\cdots w_l]) = l+1$.
Given a word $[\mathbf w]$ on $d$ letters, the \textbf{word vector} for $[w]$
is the $d$-dimensional (integral) vector whose $(i+1)$-th entry ($i=0,\cdots, d-1$) 
is the number of the letter $i$ in the characters of $[\mathbf w]$.
The following shows few examples of words $[\mathbf w_n]$ and their word vectors $\mathbf v_n$
\begin{alignat}{2}
    [\mathbf w_0] &= [0] , &&\mathbf v_0 = (1,0,0),\label{eqn:an-1} \\
    [\mathbf w_1] &= [01] , &&\mathbf v_1 = (1,1,0),\label{eqn:an-2}\\
    [\mathbf w_2] &= [0102] , &&\mathbf v_2 = (2,1,1),\label{eqn:an-3} \\
    [\mathbf w_3] &= [0102010] , &&\mathbf v_3 = (4,2,1),\label{eqn:an-4}\\
    [\mathbf w_4] &= [0102010010201] , \quad&&\mathbf v_4 = (7,4,2).\label{eqn:an}
\end{alignat}
A \textbf{subword} of $[w_0\cdots w_l]$ is a substring of the form $[w_0\cdots w_{l'}]$, $0\le l'\le l$.
For example, from Equations \eqref{eqn:an-1} - \eqref{eqn:an}, 
$[\mathbf w_i]$ is a subword for $[\mathbf w_j]$ for $i<j$.
We consider the empty word as a subword of any word.
We can concatenate words to make another word.
For example, from Equations \eqref{eqn:an-1} - \eqref{eqn:an-4}, 
\begin{equation}
	[\mathbf w_3] = [\mathbf w_2][\mathbf w_1][\mathbf w_0] = [0102010].
\end{equation} 

A \textbf{substitution on $d$ letters} is a map that transforms single-character words
$[0],\cdots, [d-1]$ to words on $d$ letters.
For example, $\sigma_i$, $i=0,1,2,3$, defined below are substitutions on $3$ letters:
\begin{alignat}{3}
   & \sigma_0([0]) = [01], &&\sigma_0([1]) = [02], \quad&&\sigma_0([2]) = [0]\label{eqn:pisot-examples-1} \\
    &\sigma_1([0]) = [12], &&\sigma_1([1]) = [2], &&\sigma_1([2]) = [0] \label{eqn:pisot-examples-2}\\
    &\sigma_2([0]) = [0102], \quad&&\sigma_1([1]) = [2], &&\sigma_1([2]) = [0] \label{eqn:pisot-examples-3}\\
    &\sigma_3([0]) = [01], &&\sigma_2([1]) = [2], &&\sigma_2([2]) = [0]\label{eqn:pisot-examples-4}
\end{alignat}
A substitution can transforms any words as follows
\begin{equation}\label{eqn:substitution}
	\sigma([w_0\cdots w_l]) = \sigma([w_0])\cdots \sigma([w_l]).
\end{equation}
With the initial word $[\mathbf w_0]=[0]$, 
any substitution $\sigma$ generates the sequence of words $[\mathbf w_n]$ recursively as follows.
\begin{equation}\label{eqn:a-n-1}
	[\mathbf w_{n+1}] = \sigma([\mathbf w_n])\textrm{ for }n\ge0.
\end{equation}
For example, $[\mathbf w_n]$, $n=0,\cdots,4$, in Equations \eqref{eqn:an-1} - \eqref{eqn:an} are the first five sequence
generated by the substitution $\sigma_0$ in Equation \eqref{eqn:pisot-examples-1}.
The substitution $\sigma_0$ in Equation \eqref{eqn:pisot-examples-1} is called the 
\textbf{Rauzy substitution}.
The sequence of words $[\mathbf w_n]$ in Equation \eqref{eqn:a-n-1}
obtained by the Rauzy substitution is called the \textbf{tribonacci words}.
Let us denote the tribonacci words as $[\mathbf a_n]$ to distinguish
them from general notation for words
The tribonacci words satisfies the following recursive formula.
\begin{equation}\label{eqn:an-recursive}
	[\mathbf a_{n+3}] = [\mathbf a_{n+2}][\mathbf a_{n+1}][\mathbf a_n]\textrm{ for }n\ge 0.
\end{equation}
Let $\mathbf v_n$ be the word vector for the word $[\mathbf w_n]$ in Equation \eqref{eqn:a-n-1}.
We can associate a unique $d\times d$ matrix $M$ for each substitution $\sigma$ that satisfies
\begin{equation}\label{eqn:an1m}
	\mathbf v_{n+1} = M\mathbf v_n\textrm{ for }n\ge 0
\end{equation}
A $d\times d$ matrix $M$ is called a \textbf{Pisot matrix} if 
its characteristic polynomial has a unique real root $\lambda$ greater than $1$ 
and all the other (complex) roots have absolute values less than $1$.
A substitution $\sigma$ is called a \textbf{Pisot substitution} if
the corresponding matrix $M$ in Equation \eqref{eqn:an1m} is a Pisot matrix.
The unique real root $\lambda$ is called a \textbf{Pisot number}.

\begin{table}
\centering
\begin{tabular}{|c|c|c|c|}
\hline
&&&\\[-1em]
\shortstack{Pisot \\ substitution} & \shortstack{Pisot \\ matrix} & 
\shortstack{characteristic \\ polynomial} & \shortstack{Pisot \\ number} \\ \hhline{|=|=|=|=|}
&&&\\[-1em]
$\sigma_0$ & $\begin{bmatrix}
    1 & 1 & 1 \\
    1 & 0 & 0 \\
    0 & 1 & 0 \end{bmatrix}$ & $1 + \lambda + \lambda^2 = \lambda^3$ & $1.8393$ \\\hline
&&&\\[-1em]
$\sigma_1$ & $\begin{bmatrix}
        0 & 1 & 1 \\
        0 & 0 & 1 \\
        1 & 0 & 0 \end{bmatrix}$ & $1 + \lambda = \lambda^3$ & $1.3247$ \\\hline
&&&\\[-1em]
$\sigma_2$ & $\begin{bmatrix}
        2 & 1 & 1 \\
        0 & 0 & 1 \\
        1 & 0 & 0
    \end{bmatrix}$ & $1 + \lambda + 2\lambda^2 = \lambda^3$ & $2.5468$ \\\hline
&&&\\[-1em]
$\sigma_3$ & $\begin{bmatrix}
        1 & 1 & 0 \\
        0 & 0 & 1 \\
        1 & 0 & 0
    \end{bmatrix}$ & $1 + \lambda^2 = \lambda^3$ & $1.4656$ \\[1em]\hline
    \end{tabular}
   \caption{Matrices associated with Pisot substitution}
   \label{tab:pisot-mat}
    \end{table}

Table \ref{tab:pisot-mat} shows the Pisot matrix and Pisot number
for the substitutions $\sigma_i$ in Equation \eqref{eqn:pisot-examples-1} - \eqref{eqn:pisot-examples-4}.
The Pisot matrix has a real eigenvector $\mathbf v_\infty$
whose eigenvalue is the Pisot number $\lambda$.
\begin{equation}
	M\mathbf v_\infty = \lambda\mathbf v_\infty.
\end{equation}
Indeed, the limit $\lim \mathbf v_n/\Vert\mathbf v_n\Vert$ is the eigenvector with the eigenvalue $\lambda$,
where $\mathbf v_n$ is the word vector for $[\mathbf w_n]$ in Equation \eqref{eqn:a-n-1}.
The hyperplane $P$ in $\mathbb R^d$ orthogonal to $\mathbf v_\infty$
is call the \textbf{contracting plane}.
For each $[\mathbf w_n]$, let $\mathbf v_l$ be the word vector for the subword of $[\mathbf w_n]$
with the length $l$. 
Let $\pi:\mathbb R^d\to P$ be the orthogonal projection and define the set $R_n$ as 
\begin{equation}\label{eqn:rn}
	R_n = \{\mathbf 0\}\cup \{\pi(\mathbf v_l)\,|\,0\le l\le L\}.
\end{equation}
We will call the set $R = \displaystyle \bigcup_{n=0}^\infty R_n$
the \textbf{Pisot domain} for $\sigma$.

\begin{figure}
    \centering
    \begin{subfigure}[b]{0.29\textwidth}
        \centering
        \includegraphics[scale=.22]{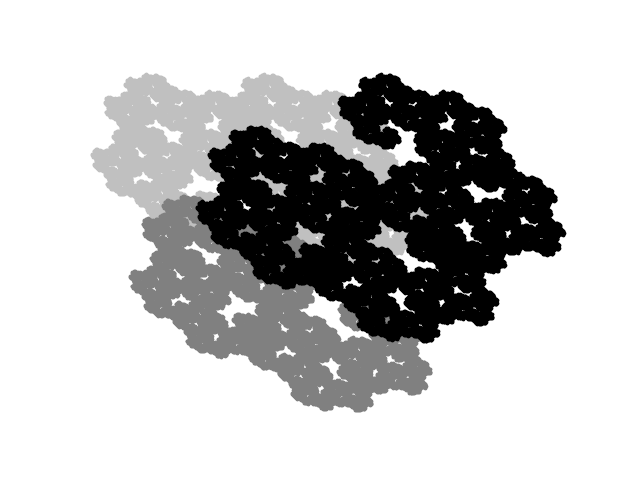}
        \caption{The Pisot domain for $\sigma_1$}
        \label{}
    \end{subfigure}
    \hfill
    \begin{subfigure}[b]{0.29\textwidth}
        \centering
        \includegraphics[scale=.22]{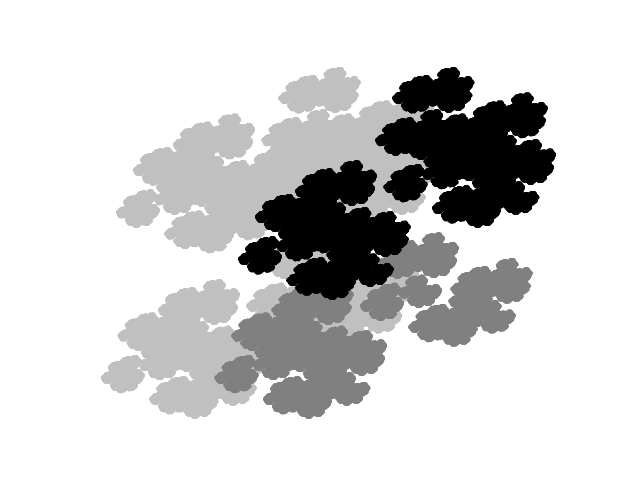}
        \caption{The Pisot domain for $\sigma_2$}
        \label{}
    \end{subfigure}
        \hfill
    \begin{subfigure}[b]{0.29\textwidth}
        \centering
        \includegraphics[scale=.22]{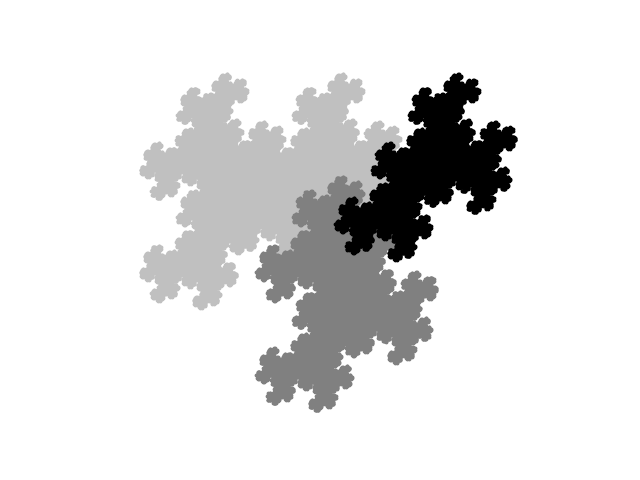}
        \caption{The Pisot domain for $\sigma_3$}
        \label{}
    \end{subfigure}
    \caption{Pisot domains for Pisot substitutions in three colors.}
    \label{fig:pisot-domains}
\end{figure}

Figure \ref{fig:rauzy-three-colors} is the Pisot domain for the Rauzy substitution $\sigma_0$,
and Figure \ref{fig:pisot-domains} are the Pisot domains for the substitutions $\sigma_i$, $i=1,2, 3$.
Here we explain how to obtain such figures. Let
\begin{equation}\label{eqn:ortho}
\mathbf e_0 = (1,0,0), \mathbf e_1 = (0,1,0), \mathbf e_2 = (0,0,1)
\end{equation}
be the standard orthogonal basis for $\mathbb R^3$.
We choose sufficiently long word $[\mathbf w_n]=[w_0\cdots w_L]$, and its word vector $\mathbf v_n$.
We approximate $\mathbf v_\infty \approx \mathbf v_n$,
to define the contracting plane $P$ to be the hyperbolic plane orthogonal to $\mathbf v_n$.
(The size of $n$ is determined by how ``dense'' the figure obtained at the end looks.)
We run the Gram--Schmidt process on the set
$\{\mathbf v_{n}/\Vert\mathbf v_{n}\Vert, \mathbf e_0, \mathbf e_1\}$
to obtain two orthonormal basis $\mathbf r_0$, $\mathbf r_1$ on $P$.
For each subword $[w_0\cdots w_l]$, $0\le l\le L$, of $[\mathbf w_n]$,
let $\mathbf v_l$ be its word vector, and define the $2$-dimensional vector $\mathbf x_l$ as follows.
\begin{equation}\label{eqn:xl}
	\mathbf x_l = (\pi(\mathbf v_l)\cdot \mathbf r_0, \pi(\mathbf v_l)\cdot \mathbf r_1).
\end{equation}
We then plot the dot at $\mathbf x_l$ for all $0\le l\le L$ 
with the color depending on the letter of the character $w_l$.

\section{Rauzy fractal construction A}\label{sec:rauzy-const-1}

Our prime concern is to understand the Rauzy fractal (Figure \ref{fig:rauzy-three-colors}) and its boundary.
We will consider the Rauzy substitution $\sigma_0$ only from now on.
We will call a $2$-dimensional point $\mathbf x$ to be \textbf{Rauzy} 
if $\mathbf x \in R_n$ for some $n\ge 0$.
Equation \eqref{eqn:an-recursive} implies that $R_n\subset R_{n+1}$.
Thus we can identify Rauzy points by the lengths of corresponding subwords uniquely.

Let $\mathbf e_0,\mathbf e_1,\mathbf e_2$ be the standard basis on $\mathbb R^3$
defined in Equation \eqref{eqn:ortho}.
Let $\mathbf u_0$, $\mathbf u_1$, and $\mathbf u_2$ be the $2$-dimensional vectors defined by
\begin{equation}\label{eqn:ui}
	\mathbf u_i = (\pi(\mathbf e_i)\circ \mathbf r_0, \pi(\mathbf e_i)\circ \mathbf r_1),\textrm{ for } i=0,1,2,
\end{equation}
where $\mathbf r_0, \mathbf r_1$ are orthonormal basis on the contracting plane $P$
for the Rauzy substitution.

Let $\mathbf b_n$ be the Rauzy point for the tribonacci words $[\mathbf a_n]$.
The first four $\mathbf b_n$ are
\begin{align}
 \mathbf b_0 &= \mathbf u_0,  \\
 \mathbf b_1 &= \mathbf u_0 + \mathbf u_1,  \\
 \mathbf b_2 &= 2\mathbf u_0 + \mathbf u_1 + \mathbf u_2, \\
 \mathbf b_3 &= 4\mathbf u_0 + 2\mathbf u_1 + \mathbf u_2, \\
 \mathbf b_4 &= 7\mathbf u_0 + 4\mathbf u_1 + 2\mathbf u_2.
\end{align}
For a semantic reason, let us denote $\mathbf b_i^{(j)} = \mathbf b_{3i+j}$.
The \textbf{A-layer at the level $i$} is the set $V_i$ of Rauzy poitns defined inductively as follows:
\begin{align}
V_{-1} &= \{\mathbf 0\}, \textrm{ and for }i\ge 0,\label{eqn:vi-layer-1} \\
V_i^{(0)} &= \{\mathbf x + \mathbf b_i^{(0)}\,|\, \mathbf x\in \bigcup_{j=-1}^{i-1}V_j\}, \label{eqn:vi-layer-3} \\
V_i^{(1)} &= \{\mathbf x + \mathbf b_i^{(1)},\mathbf x + \mathbf b_i^{(0)}+\mathbf b_i^{(1)}\,|\, \mathbf x\in \bigcup_{j=-1}^{i-1}V_j\}, \label{eqn:vi-layer-4} \\
V_i^{(2)} &= \{\mathbf x + \mathbf b_i^{(2)}, \mathbf x + \mathbf b_i^{(0)} + \mathbf b_i^{(2)}, \mathbf x + \mathbf b_i^{(1)} + \mathbf b_i^{(2)}\,|\, \mathbf x\in \bigcup_{j=-1}^{i-1}V_j\},\label{eqn:vi-layer-5} \\
V_i &= \bigcup_{j=0}^2 V_i^{(j)}.\label{eqn:vi-layer-6}
\end{align}
We have the following result.
\begin{thm}\label{thm:main-1}
The set of all Rauzy points is the union of all A-layers at all levels.
\end{thm}

\begin{figure}
\centering
\begin{subfigure}[t]{0.475\textwidth}
    \centering
    \includegraphics[scale=.3]{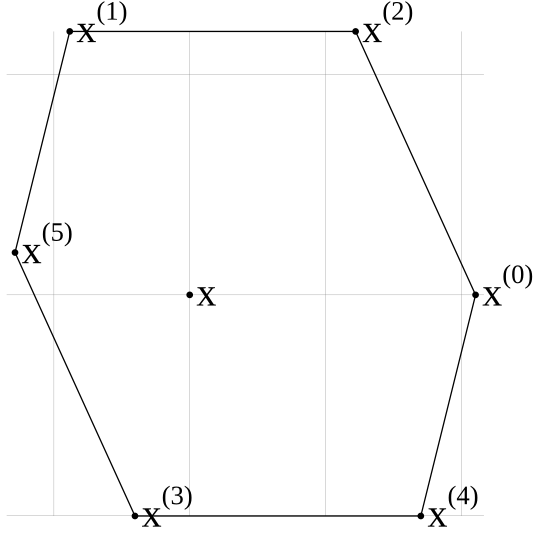}
    \caption{The six children points $\mathbf x^{(j)}$ form a cell of their parent point $\mathbf x$.}
    \label{fig:rauzy-1}
\end{subfigure}
\hfill 
\begin{subfigure}[t]{0.475\textwidth}
\centering
\includegraphics[scale=.2]{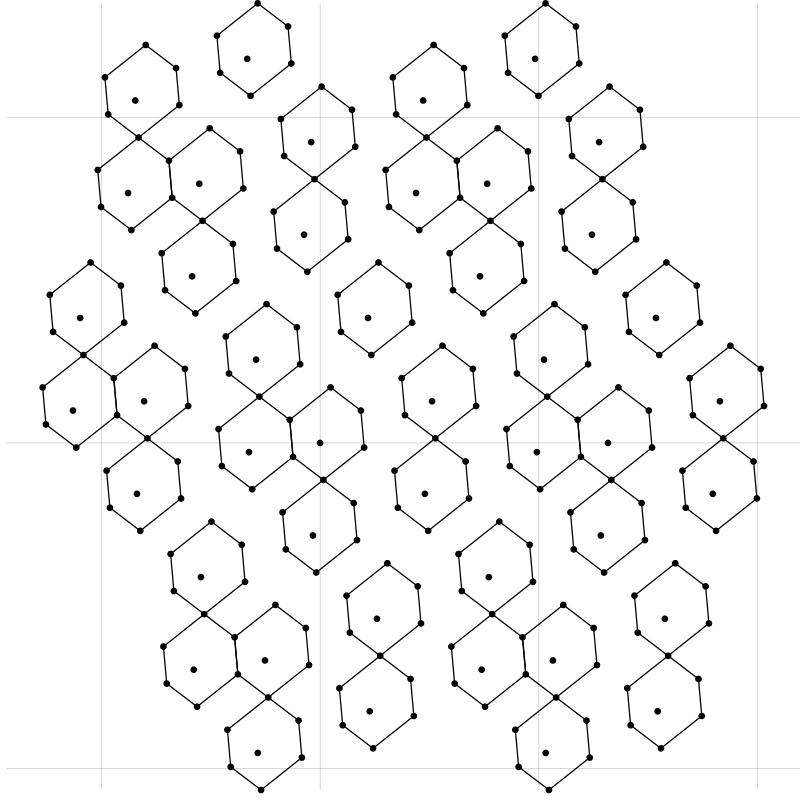}
\caption{The cells at each level are uniformly distributed.}
\label{fig:rauzy-4}
\end{subfigure}
\caption{Rauzy points in Rauzy fractal. The $x$, $y$-axis are equally scaled.}
\label{fig:rauzy-boundary-1}
\end{figure}

Before we prove the theorem, let us show how we can use this theorem 
to construct Rauzy fractal and its boundary.
For any $\mathbf x\in V_{i-1}$, the following six points are in $V_i$:
\begin{alignat}{2}
\mathbf x^{(0)} &= \mathbf x + \mathbf b_i^{(0)} &&\in V_i^{(0)},\label{eqn:points-vi-1}\\
\mathbf x^{(1)} &= \mathbf x + \mathbf b_i^{(1)} &&\in V_i^{(1)},\label{eqn:points-vi-2}\\
\mathbf x^{(2)} &= \mathbf x + \mathbf b_i^{(0)} + \mathbf b_i^{(1)} &&\in V_i^{(1)},\label{eqn:points-vi-3}\\
\mathbf x^{(3)} &= \mathbf x + \mathbf b_i^{(2)} &&\in V_i^{(2)},\label{eqn:points-vi-4}\\
\mathbf x^{(4)} &= \mathbf x + \mathbf b_i^{(0)} + \mathbf b_i^{(2)} &&\in V_i^{(2)},\label{eqn:points-vi-5}\\
\mathbf x^{(5)} &= \mathbf x + \mathbf b_i^{(1)} + \mathbf b_i^{(2)} &&\in V_i^{(2)}.\label{eqn:points-vi-6} 
\end{alignat}
These points forms a hexagon around $\mathbf x$ (Figure \ref{fig:rauzy-1}).
We will call this hexagon as the \textbf{cell} of $\mathbf x$.
The point $\mathbf x^{(j)}$, $j=0,\cdots,5$, is a \textbf{child} of its \textbf{parent} $\mathbf x$.
Theorem \ref{thm:main-1} implies that Rauzy fractal can be filled recursively 
by tracing the genealogy of the Rauzy points, starting from the origin $\mathbf 0$.
To plot the boundary of Rauzy fractal, 
we trace only the descendents that lie outside of the cells of their ancestors.
These points are called the \textbf{boundary} points.
(To speed up the computation, we can only check whether a points lie inside of the cell of its grandparent.)
At each level of layers, we produce the children only for the boundary points,
and they are clustered into cells that are uniformly distributed.
Figure \ref{fig:rauzy-4} shows all points in A-layers up to the level $3$. 
In fact, they are the first $274$ Rauzy points, including the origin.
The cells are only drawn for the points in A-layer at the level $3$.
The interior points in cells are all the points in A-layer up to the level $2$.
We can see that there is no conflict on the geneology of points:
no two points in the same level of A-layer never enjoy a child-parent relationship.

\begin{figure}
    \centering
    \begin{subfigure}[b]{0.49\textwidth}
        \centering
        \includegraphics[scale=.18]{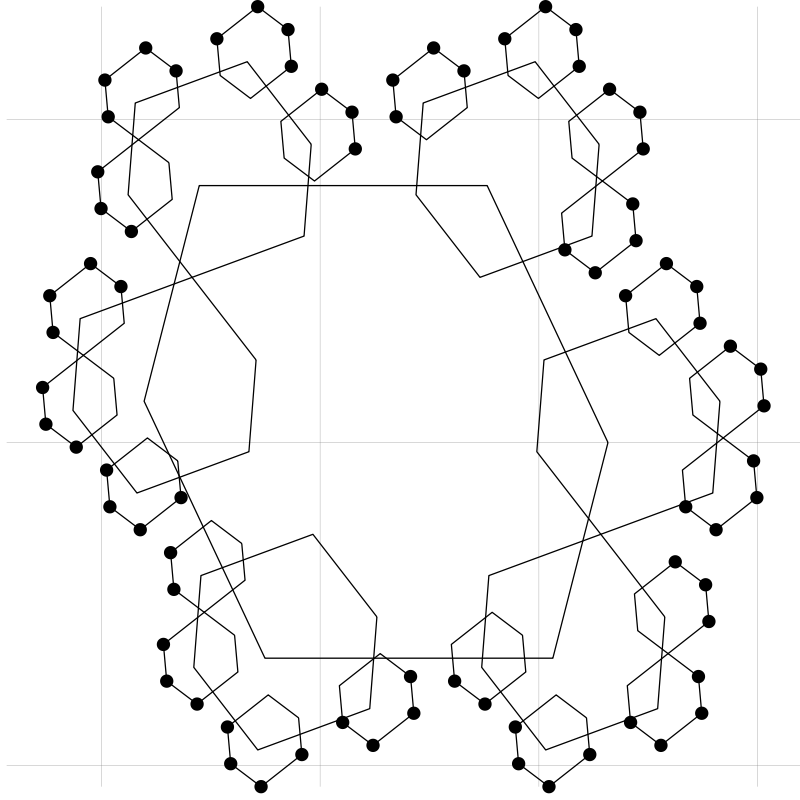}
        \caption{The boundary points and cells up to the level $3$.}
        \label{fig:rauzy-boundary-2}
    \end{subfigure}
    \begin{subfigure}[b]{0.49\textwidth}
        \centering
        \includegraphics[scale=.18]{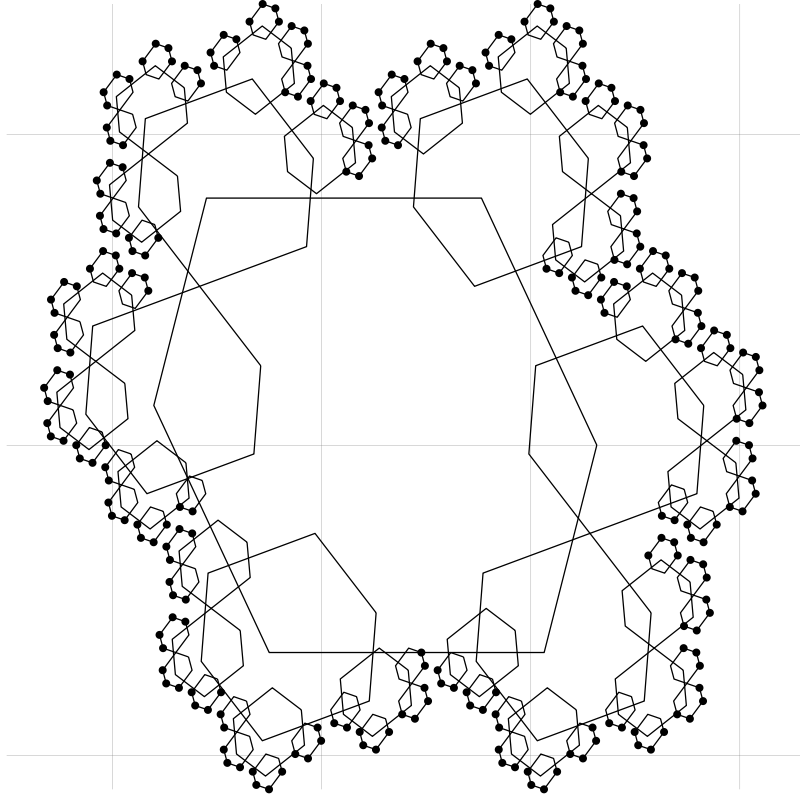}
        \caption{The boundary points and cells up to the level $4$.}
        \label{fig:rauzy-boundary-3}
    \end{subfigure}
	\hfill
    \begin{subfigure}[b]{0.49\textwidth}
        \centering
        \includegraphics[scale=.18]{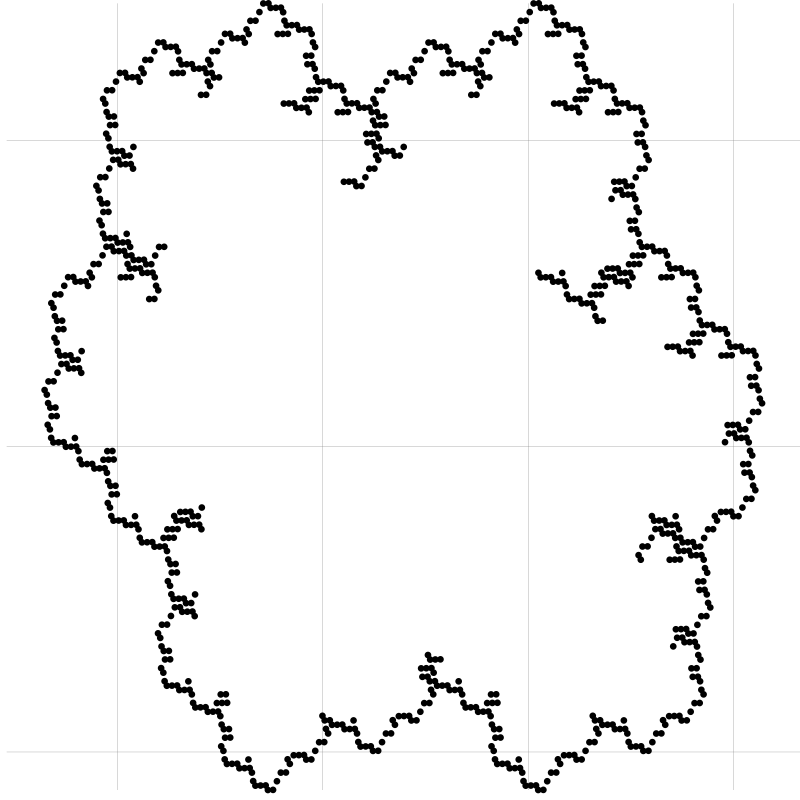}
        \caption{The boundary points up to the layer $5$.}
        \label{fig:rauzy-boundary-5}
    \end{subfigure}
    \begin{subfigure}[b]{0.49\textwidth}
        \centering
        \includegraphics[scale=.18]{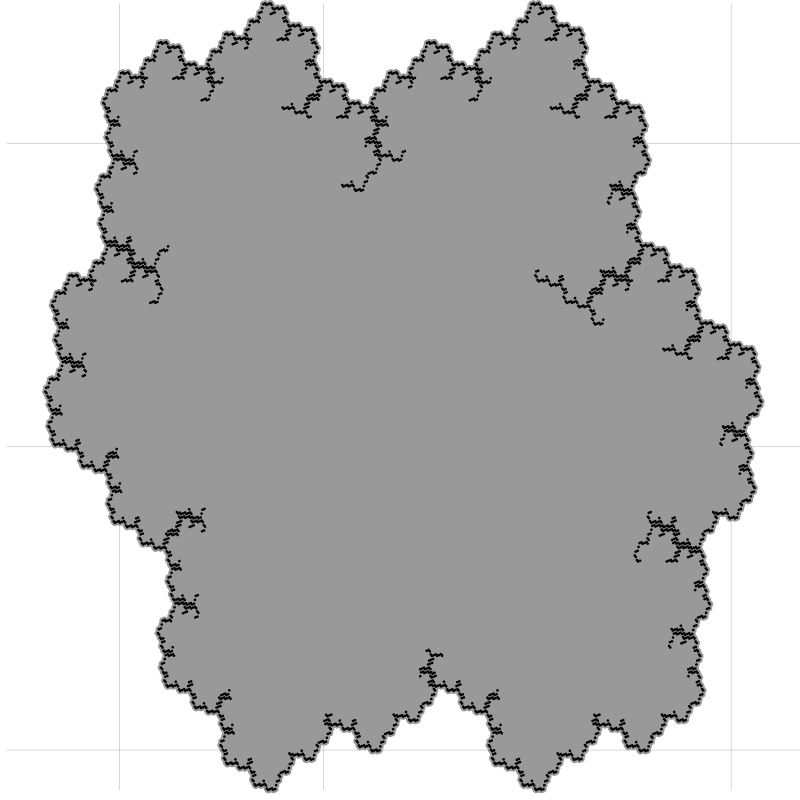}
        \caption{The boundary of Rauzy fractal with its interior.}
        \label{fig:rauzy-boundary-interior}
    \end{subfigure}
    \caption{The boundary Rauzy points in A-layers and the cells.}
    \label{fig:rauzy-boundary}
\end{figure}

Figure \ref{fig:rauzy-boundary-2} shows the Rauzy points (filled-dots) 
in A-layers at the level $3$.
The hexagons are the cells of boundary points in the previous layers.
We can plot only the boundary points at each layer to determine the boundary of Rauzy fractal. 
Figure \ref{fig:rauzy-4} shows all boundary points in the A-layer at the level $4$
and the cells of all boundary points in the previous layers.
One can see that the boundary point at the top layer forms the boundary of the Rauzy fractal.
Figure \ref{fig:rauzy-boundary-5} shows the boundary points at the level $5$,
and Figure \ref{fig:rauzy-boundary-interior} shows them with Rauzy fractal.
One might wonder why there seems to be several \textit{estuaries} (rivers that meet the ocean),
if we put an analogy of Rauzy fractal as the ``land" and its complement as an ``ocean".
This phenomenon does not mean that the Rauzy fractal has a vacuous ``valley".
It is a consequence of discarding non-boundary points at each level.
To reduce the computational complexity, it is inevitable to disregard some cells
when checking the interior points. 
As we observed in Figure \ref{fig:rauzy-4}, all Rauzy points are almost-uniformly distributed.
The shape of the boundary of the Rauzy fractal in Figure \ref{fig:rauzy-boundary-5}
is optimal in our setting.
To remove all \textit{estuaries}, we should consider all cells at each level of A-layers.
This is essentially the same as the conventional way of drawing the Rauzy fractal,
and requires the same amount of computational complexity.

Let us prove Theorem \ref{thm:main-1}. We will prove the following holds for all $i\ge 0$:
\begin{equation}\label{eqn:vir}
	V_i = R_{3i+3} - \{\mathbf b_{3i+3}\}.
\end{equation}
Since every Rauzy point $\mathbf x_l$ is uniquely identified with the length of corresponding substring $[\mathbf w_l]$ of a tribonacci word,
we define the \textbf{length} of the Rauzy point $\mathbf x_l$ as 
the length of its subword $[\mathbf w_l]$:
\begin{equation}
	L(\mathbf x_l) = L([\mathbf w_l]).
\end{equation}
The following lemma is the key to the proof of Equation \eqref{eqn:vir}.
\begin{lem*}\label{lem:xbn}
If $\mathbf x$ is a Rauzy point satisfying 
\begin{equation}\label{eqn:lwl}
	L(\mathbf x) \le L(\mathbf b_{n-1}) + L(\mathbf b_{n-2}),
\end{equation}
then $\mathbf x + \mathbf b_n$ is also a Rauzy point.
\end{lem*}
\begin{proof}
Let $[\mathbf w]$ be the word corresponding to the Rauzy point $\mathbf x$.
From the condition of the lemma, $[\mathbf w]$ is a substring of the word
$[\mathbf a_{n-1}][\mathbf a_{n-2}]$.
Thus the word $[\mathbf a_n][\mathbf w]$ is a substring of $[\mathbf a_{n+1}]$,
and it contains $[\mathbf a_n]$.
Therefore $\mathbf x + \mathbf b_n$ is a Rauzy point.
\end{proof}

Let us prove Equation \eqref{eqn:vir} by the induction.
Obviously, $V_{-1} = R_0 - \{\mathbf b_0\} = \{\mathbf 0\}$.
We get the first six Rauzy points by letting $\mathbf x=\mathbf 0$ in
Equations \eqref{eqn:points-vi-1}-\eqref{eqn:points-vi-6}.
Thus $V_0 = R_3 -\{\mathbf b_3\}$.
This shows that Equation \eqref{eqn:vir} holds for $i=-1,0$.

Let us assume that Equation \eqref{eqn:vir} holds for $i-2$ and $i-1$ for $i\ge 1$. 
We will show that each set $V_i^{(j)}$ is a subset of $R_{3i+3} -\{\mathbf b_{3i+3}\}$,
and all Rauzy points $\mathbf x$ satisfying $L(\mathbf b_{3i})\le L(\mathbf x)<L(\mathbf b_{3i+3})$ appears in one of $V_i^{(j)}$.

\begin{figure}
\centering
    \includegraphics[scale=.25]{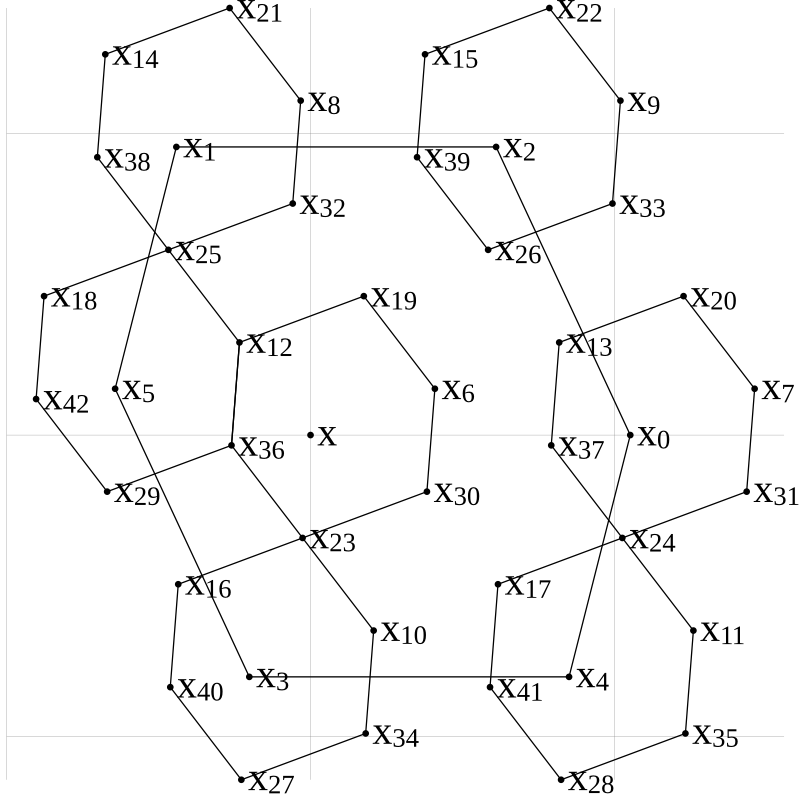}
    \caption{The first $43$ Rauzy points in A-layers up to the level $2$
    are shown with numbered indices. The cells are represented by hexagons.}\label{fig:rauzy-2}
\end{figure}

To help understanding, we will use Figure \ref{fig:rauzy-2},
which shows the first $51$ Rauzy points with the origin $\mathbf x$.
It shows the general configurations of children and grand children of $\mathbf x$. 
We can see that some points have two parents.
For any choice of $\mathbf x\in V_{i-2}$, the ancestries of all points are identical.
The inclusions and exclusions of points to the cells of their ancesters 
may differ by the choice of $\mathbf x$, but our proof does not depend on it.
We assume $\mathbf x\in V_{i-2}$ and 
$\mathbf x_0,\cdots,\mathbf x_5\in V_{i-1}$ be the children of $\mathbf x$
obtained by Equations \eqref{eqn:points-vi-1}-\eqref{eqn:points-vi-6}.
We also assume that $\mathbf x_0,\cdots,\mathbf x_5$ are Rauzy points.
We will show that
\begin{align}
	\mathbf x_6,\cdots,\mathbf x_{12} &\in V_i^{(0)},\label{eqn:x6-1}\\
	\mathbf x_{13},\cdots,\mathbf x_{25} &\in V_i^{(1)},\label{eqn:x6-2}\\
	\mathbf x_{26},\cdots,\mathbf x_{42} &\in V_i^{(2)}.\label{eqn:x6-3}
\end{align}

The points in Equation \eqref{eqn:x6-1} are obtained by adding $\mathbf b_i^{(0)}$:
\begin{equation}\label{eqn:xj7}
	\mathbf x_6 = \mathbf x + \mathbf b_i^{(0)},\quad
	\mathbf x_{j+7} = \mathbf x_j + \mathbf b_i^{(0)}\textrm{ for }0\le j \le 5.
\end{equation}
The assumption $\mathbf x\in V_{i-1}$ implies $L(\mathbf x) < L(\mathbf b_{i-1}^{(0)})$.
Thus we can use Equations \eqref{eqn:points-vi-1}-\eqref{eqn:points-vi-5} 
to show that $\mathbf x$ and $\mathbf x_0,\cdots,\mathbf x_4$ satisfies
the inequality \eqref{eqn:lwl}:
\begin{alignat}{2}
L(\mathbf x) &< L(\mathbf b_{3i-3}) &&< L(\mathbf b_{3i-2}), \label{eqn:lw-1}\\
L(\mathbf x_0) &< L(\mathbf b_{3i-3}) + L(\mathbf b_{3i-3})&&< L(\mathbf b_{3i-1}), \label{eqn:lw-2}\\
L(\mathbf x_1) &< L(\mathbf b_{3i-3}) + L(\mathbf b_{3i-2})&&< L(\mathbf b_{3i-1}), \label{eqn:lw-3}\\
L(\mathbf x_2) &< L(\mathbf b_{3i-3}) + L(\mathbf b_{3i-2}) + L(\mathbf b_{3i-3})&&<L(\mathbf b_{3i-1}) + L(\mathbf b_{3i-2}), \label{eqn:lw-4}\\
L(\mathbf x_3) &< L(\mathbf b_{3i-3}) + L(\mathbf b_{3i-1})&&< L(\mathbf b_{3i-1}) + L(\mathbf b_{3i-2}), \label{eqn:lw-5}\\
L(\mathbf x_4) &< L(\mathbf b_{3i-3}) + L(\mathbf b_{3i-1}) + L(\mathbf b_{3i-3})&&\le L(\mathbf b_{3i-1}) + L([\mathbf a_{3i-2}]). \label{eqn:lw-6}
\end{alignat}
Therefore, $\mathbf x_6,\cdots, \mathbf x_{11}$ are Rauzy points by the lemma.
Equation \eqref{eqn:points-vi-6} shows that
\begin{equation}\label{eqn:lw-7}
L(\mathbf x_5) \le L(\mathbf b_{3i-6}) + L(\mathbf b_{3i-1}) + L(\mathbf b_{3i-2}).
\end{equation}
This does not immediately imply the inequality \eqref{eqn:lwl}.
Meanwhile, $\mathbf x_{12}$ also satisfies
\begin{equation}
	\mathbf x_{12} = \mathbf x + \mathbf b_{i-1}^{(1)} + \mathbf b_{i-1}^{(2)} + \mathbf b_i^{(0)}
	= \mathbf x + \mathbf b_i^{(1)}.
\end{equation}
Therefore $\mathbf x_{12}$ is also a Rauzy point by the lemma.
The point $\mathbf x_{12}$ is the longest Rauzy point among its siblings.
Since the longest length for $\mathbf x$ is $L(\mathbf b_{i-1}^{(0)})-1$, 
\begin{equation}
	\max L(\mathbf x_{12}) = L(\mathbf b_{i-1}^{(0)}) + L(\mathbf b_i^{(1)}) - 1 < L(\mathbf b_i^{(2)}).
\end{equation}
Thus we proved that $V_i^{(0)}\subset R_{3i+2}\subset R_{3i+3}-\{\mathbf b_{3i+3}\}$.

Next, we show that points in Equation \eqref{eqn:x6-2} are Rauzy points.
On the other hands, they are obtained by adding $\mathbf b_i^{(1)}$ to the previous points:
\begin{equation}\label{eqn:xj13}
	\mathbf x_{j+13} = \mathbf x_j + \mathbf b_i^{(1)},\quad 0\le j \le 12.
\end{equation}
From Equations \eqref{eqn:points-vi-2} - \eqref{eqn:points-vi-3}, we can easily check that 
\begin{equation}
	L(\mathbf x_j) \le L(\mathbf b_{i-1}^{(2)}) + L(\mathbf b_i^{(0)})\textrm{ for }0\le j\le 9.
\end{equation}
However, this inequality is not immediate obvious for the following three points.
\begin{align}
	L(\mathbf x_{10}) &= L(\mathbf x) + L(\mathbf b_{i-1}^{(2)}) + L(\mathbf b_i^{(0)}), \\
	L(\mathbf x_{11}) &= L(\mathbf x) + L(\mathbf b_{i-1}^{(0)}) + L(\mathbf b_{i-1}^{(2)}) + L(\mathbf b_i^{(0)}), \\
	L(\mathbf x_{12}) &= L(\mathbf x) + L(\mathbf b_{i-1}^{(1)}) + L(\mathbf b_{i-1}^{(2)}) + L(\mathbf b_i^{(0)}).
\end{align}
In fact, these points are obtained by adding $\mathbf b_i^{(2)}$ to the previous Rauzy points.
\begin{align}
	\mathbf x_{23} &= \mathbf x + \mathbf b_i^{(2)}, \\
	\mathbf x_{24} &= \mathbf x + \mathbf b_{i-1}^{(0)} + \mathbf b_i^{(2)}, \\
	\mathbf x_{25} &= \mathbf x + \mathbf b_{i-1}^{(1)} + \mathbf b_i^{(2)}.\label{eqn:x2325-3}
\end{align}
Therefore $\mathbf x_{23}, \mathbf x_{24}, \mathbf x_{25}$ are Rauzy points by the lemma.
The maximum length of $\mathbf x_{25}$ is 
\begin{equation}
	\max L(\mathbf x_{25}) =  L(\mathbf b_{i-1}^{(0)}) + L(\mathbf b_{i-1}^{(1)}) + L(\mathbf b_i^{(2)}) - 1,
\end{equation}
and it is strictly less than $L(\mathbf b_{i+1}^{(0)})$.
Thus $V_i^{(1)}\subset R_{3i+3}-\{\mathbf b_{3i+3}\}$.

Using Equations \eqref{eqn:points-vi-4} - \eqref{eqn:points-vi-6},
we can apply the lemma to the points in Equation \eqref{eqn:x6-3} without any exception.
The maximum length of the longest Rauzy point is $\mathbf x_{42}$ is
\begin{align}
	\max L(\mathbf x_{42}) 
	&= L(\mathbf b_{i-1}^{(0)}) + L(\mathbf b_{i-1}^{(1)}) + L(\mathbf b_{i-1}^{(2)})  
	+ L(\mathbf b_i^{(1)}) + L(\mathbf b_i^{(2)}) - 1\nonumber\\
	&=L(\mathbf b_i^{(0)}) + L(\mathbf b_i^{(1)}) + L(\mathbf b_i^{(2)}) - 1 \nonumber \\
	&= L(\mathbf b_{i+1}^{(0)}) -1.
\end{align}
Therefore, $V_i^{(2)}\subset R_{3i+3} -\{\mathbf b_{3i+3}\}$.

So far, we showed that
\begin{equation}
	V_i^{(0)} \cup V_i^{(1)} \cup V_i^{(2)} \subset (R_{3i+3} - \{\mathbf b_{3i+3}\}) - (R_{3i} - \{\mathbf b_{3i}\})
\end{equation}
Let us show that the opposite holds too.
First, suppose that $\mathbf x$ is a Rauzy point satisfying
\begin{equation}
	L(\mathbf b_i^{(0)}) \le L(\mathbf x) < L(\mathbf b_i^{(1)}).
\end{equation}
Then the word $[\mathbf w]$ corresponding to $\mathbf x$ is of the form
\begin{equation}\label{eqn:wa3i}
	[\mathbf w] = [\mathbf a_{3i}]\mathbf [\mathbf w']
\end{equation}
where $[\mathbf w']$ is a substring of the word $[\mathbf a_{3i-1}][\mathbf a_{3i-2}]$.
Therefore the word $[\mathbf w']$ is a subword in 
the tribonacci word $[\mathbf a_{3i}]$,
and the corresponding Rauzy points $\mathbf x'$ lies in $V_{i-1}$.
Since Equation \eqref{eqn:wa3i} implies
\begin{equation}
\mathbf x = \mathbf x' + \mathbf b_i^{(0)},
\end{equation}
we havev $\mathbf x\in V_i^{(0)}$.
This is the point $\mathbf x^{(0)}$ in Equation \eqref{eqn:points-vi-1}

Next, suppose that $\mathbf x$ is a Rauzy point satisfying
\begin{equation}
	L(\mathbf b_i^{(1)})\le L(\mathbf x) < L(\mathbf b_i^{(2)}).
\end{equation} 
Then the word $[\mathbf w]$ corresponding $\mathbf x$ is of the form
\begin{equation}
	[\mathbf w] =[\mathbf a_{3i+1}][\mathbf w']
\end{equation}
where $[\mathbf w']$ is a substring of $[\mathbf a_{3i}][\mathbf a_{3i-1}]$.
If $L([\mathbf w']) < L([\mathbf a_{3i}])$, then the 
Rauzy point $\mathbf x'$ corresponding $[\mathbf w']$ lies in $V_{i-1}$, and
\begin{equation}
	\mathbf x = \mathbf x' + \mathbf b_i^{(1)}.
\end{equation}
This implies that $\mathbf x\in V_i^{(1)}$.
In fact, this point is $\mathbf x^{(1)}$ in Equation \eqref{eqn:points-vi-2}.
On the other hand, if $L([\mathbf a_{3i}])\le L([\mathbf w'])$, then $[\mathbf w']$ is of the form
\begin{equation}
	\mathbf w' = [\mathbf a_{3i}][\mathbf w''],
\end{equation}
where $[\mathbf w'']$ is a substring of $[\mathbf a_{3i-1}]$.
In this cacse, the Rauzy point $\mathbf x''$ for $\mathbf w''$ lies in $V_{i-1}^{(2)}\subset V_{i-1}$.
Since
\begin{equation}
	\mathbf x = \mathbf x'' + \mathbf b_i^{(0)} + \mathbf b_i^{(1)},
\end{equation}
we have $\mathbf x\in V_i^{(1)}$. This point is $\mathbf x^{(2)}$ in 
Equation \eqref{eqn:points-vi-3}.

Finally, suppose that $\mathbf x$ satisfies
\begin{equation}
	L(\mathbf b_i^{(2)}) \le L(\mathbf x) < L(\mathbf b_{i+1}^{(0)}).
\end{equation}
Then the Rauzy word $\mathbf w$ for $\mathbf x$ is of the form
\begin{equation}
	[\mathbf w] =[\mathbf a_{3i+2}][\mathbf w']
\end{equation}
where $[\mathbf w']$ is a substring of $[\mathbf a_{3i+1}][\mathbf a_{3i}]$.
If $[\mathbf w']$ is a substring of $[\mathbf a_{3i+1}]$, then
$\mathbf x$ is the point $\mathbf x^{(4)}$ or $\mathbf x^{(5)}$
in Equations \eqref{eqn:points-vi-4} and \eqref{eqn:points-vi-5}.
If $\mathbf w'$ is a substring that is longer than $[\mathbf a_{3i+1}]$, 
then $\mathbf x$ is the point $\mathbf x^{(6)}$ in Equation \eqref{eqn:points-vi-6}.
In either case, we have $\mathbf x \in V_i^{(2)}$.

\section{Rauzy fractal construction B} \label{sec_alternative_construction}

Now we present yet another way to construct the Rauzy fractal. 
Let $\mathbf u_i$ be the $2$-dimensional vector defined in Equation \eqref{eqn:ui}.
We define the vector $\mathbf s_i^{(j)}$ for $i\ge 0$ and $j=0,\ldots,6$ as follows:
\begin{align}
\mathbf s_0^{(0)} &= \mathbf s_0^{(2)} = \mathbf s_0^{(4)} = \mathbf s_0^{(6)} = \mathbf u_0,\label{eq-sn-1}\\
\mathbf s_0^{(1)} &= \mathbf s_0^{(5)} = \mathbf u_1,\quad \mathbf s_0^{(3)} = \mathbf u_2, \textrm{ and for } i\ge0,\label{eq-sn-2}\\
\mathbf s_{i+1}^{(0)} &= \mathbf s_i^{(0)} + \mathbf s_i^{(1)} + \mathbf s_i^{(2)} + \mathbf s_i^{(3)} + \mathbf s_i^{(4)} + \mathbf s_i^{(5)} + \mathbf s_i^{(6)}, \label{eq-sn-3}\\
\mathbf s_{i+1}^{(1)} &= \mathbf s_i^{(0)} + \mathbf s_i^{(1)} + \mathbf s_i^{(2)} + \mathbf s_i^{(3)} + \mathbf s_i^{(4)} + \mathbf s_i^{(5)},\label{eq-sn-4} \\
\mathbf s_{i+1}^{(2)} &= \mathbf s_i^{(0)} + \mathbf s_i^{(1)} + \mathbf s_i^{(2)} + \mathbf s_i^{(3)} + \mathbf s_i^{(4)} + \mathbf s_i^{(5)} + \mathbf s_i^{(6)},\label{eq-sn-5}\\
\mathbf s_{i+1}^{(3)} &= \mathbf s_i^{(0)} + \mathbf s_i^{(1)} + \mathbf s_i^{(2)} + \mathbf s_i^{(3)} \label{eq-sn-6}\\
\mathbf s_{i+1}^{(4)} &= \mathbf s_i^{(0)} + \mathbf s_i^{(1)} + \mathbf s_i^{(2)} + \mathbf s_i^{(3)} + \mathbf s_i^{(4)} + \mathbf s_i^{(5)} + \mathbf s_i^{(6)}, \label{eq-sn-7}\\
\mathbf s_{i+1}^{(5)} &= \mathbf s_i^{(0)} + \mathbf s_i^{(1)} + \mathbf s_i^{(2)} + \mathbf s_i^{(3)} + \mathbf s_i^{(4)} + \mathbf s_i^{(5)}, \label{eq-sn-8} \\
\mathbf s_{i+1}^{(6)} &= \mathbf s_i^{(0)} + \mathbf s_i^{(1)} + \mathbf s_i^{(2)} + \mathbf s_i^{(3)} + \mathbf s_i^{(4)} + \mathbf s_i^{(5)} + \mathbf s_i^{(6)}.\label{eq-sn-9}
\end{align}
Let $W_{-1}= \{\mathbf 0\}$ and for each $i\ge 0$, define the set $W_i$ as follows:
\begin{equation}\label{eqn:wiwi-1}
W_i = \{ \mathbf x + \sum_{k=0}^j\mathbf s_i^{(k)} \,|\, \mathbf x\in \bigcup_{k=-1}^{i-1}W_k, 0\le j\le 6\}.
\end{equation}
Then we have the following result.
\begin{thm}\label{thm:main-2}
The set of all Rauzy points is the union all $W_i$.
\end{thm}

\begin{figure}
    \centering
    \begin{subfigure}[b]{0.475\textwidth}
        \centering
        \includegraphics[scale=.3]{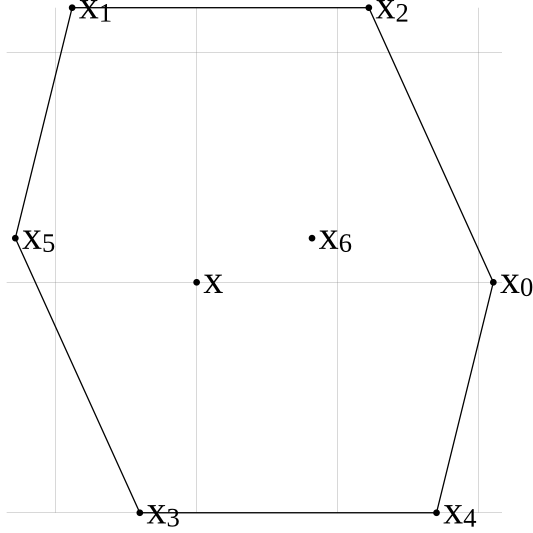}
        \caption{Rauzy points in the B-layers at the levels up to $1$.}
        \label{fig:rauzy-1-type-2}
    \end{subfigure}
    \hfill
    \begin{subfigure}[b]{0.475\textwidth}
        \centering
        \includegraphics[scale=.2]{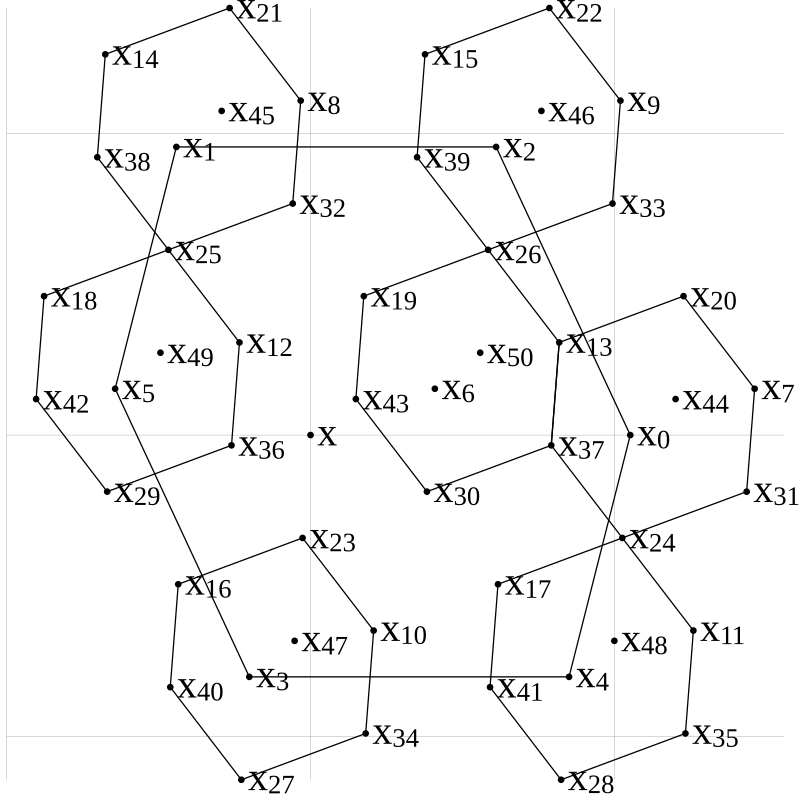}
        \caption{Rauzy points in the B-layer at the levels up to $2$.}
        \label{fig:rauzy-2-type-2}
    \end{subfigure}
    \caption{Rauzy points generated by sets $W_i$.}
    \label{fig:rauzy-boundary-seq}
\end{figure}

Before we procede to the proof, let us visualize how Theorem \ref{thm:main-2}
contruct Rauzy fractal and its boundary.
To help understanding, we will use the terminologies that we defined in previous sections
here in similar ways.
The set $W_i$ is the \textbf{$i$-layer of type B},
the point $\mathbf x + \sum\mathbf s_i^{(k)}$ is the \textbf{child} of $\mathbf x$,
and the hexagon bounded by the outer six children is the \textbf{cell} of $\mathbf x$.
Figure \ref{fig:rauzy-1-type-2} and \ref{fig:rauzy-2-type-2}
shows the points in the $1$-layer and $2$-layer respectively.
Unlike to the layers of type A,
each point in the $(i-1)$-layer of type A is the parent to the \textit{seven} children points in the $i$-layer of type A.

The morphic pattern in Equations \eqref{eq-sn-1} - \eqref{eq-sn-9} is obtained by the following observation.
The tribonacci word $[\mathbf a_6]$ is obtained from concatenating copies of $[\mathbf a_3] = [0102010]$ 
with trimmings at the right end:
\begin{equation}\label{eqn:a6-patt}
	[\mathbf a_6] = [\underbrace{0102010}_{0\textrm{ trim}}] \Big| [\underbrace{010201}_{1\textrm{ trim}}] \Big| [\underbrace{0102010}_{0\textrm{ trim}}] \Big| [\underbrace{0102}_{\underline{3}\textrm{ trims}}] \Big| [\underbrace{0102010}_{0\textrm{ trim}}] \Big| [\underbrace{010201}_{1\textrm{ trim}}] \Big| [\underbrace{0102010}_{0\textrm{ trim}}]
\end{equation}
The numbers of trims follows the pattern of letters in $[\mathbf a_3]$,
except for the case of letter $2$: we trim $3$ letters instead of $2$.
That means, we follow the sequence $0102010$ from left to right
and trim the last $0$, $1$, or $3$ characters from $[\mathbf a_3]$ to obtain substrings 
and concatenating them.
The tribonacci $[\mathbf a_9]$ satisfies the same pattern.
Each substring in $[\mathbf a_6]$ separated by $|$ in Equation \eqref{eqn:a6-patt} 
become a new unit for the next trimming.
For example, the first substring of $[\mathbf a_9]$ 
is obtained by trimming the last $[0102010]$ from $[\mathbf a_6]$ (not $[0]$).

The exception at the letter $2$ is essential to obtain the Rauzy fractal.
Because of the exception at $2$, we will call this pattern as
\textit{pseudo self-replicating}.
Interestingly, removing this exception seems to give discrete tiling of $\mathbb R^2$ with a compact domain.
Tilings obtained by fully self-replicating pattern 
will be explored in the next section.

Theorem \ref{thm:main-2} implies that
we can produce Rauzy points by the simple pattern from the sequence $0102010$
using the three vectors $\mathbf u_0$, $\mathbf u_1$, $\mathbf u_2$.
The construction B also gives the boundary of Rauzy fractal.
In fact, it gives the same figures as in Figure \ref{fig:rauzy-boundary}.
We take only the boundary points at each B-layers and produce children points at the next layer.
The only difference between the A-layers and the B-layers is that the
B-layers contains children that are lie interior to the cells of their parent. 
For example the point $\mathbf x_6$ in Figure \ref{fig:rauzy-1-type-2}
lies at the level $1$ for the A-layer, but it appears at the level $0$ for the B-layer.

The upshot for using the construction B is that the ingredients to produce
Rauzy fractal and its boundary are essentially the three vectors $\mathbf u_0,\mathbf u_1,\mathbf u_2$
and the self-replicating sequence $0102010$.
Although we need an exception at the letter $2$ to draw the Rauzy fractal,
it is much simpler than having six rules in Equations \eqref{eqn:points-vi-1} - \eqref{eqn:points-vi-6}.
Moreover, the construction A requires to prepare the Rauzy points $\mathbf b_i^{(j)}$,
where as the construction B requires none.

Now let us prove Theorem \ref{thm:main-2}. We will show that the following holds for all $i\ge 0$:
\begin{equation}\label{eqn:viwi}
	V_i \subset W_i \subset V_{i+1}^{(0)}.
\end{equation}
In particular, we will show that the following equalities hold for all $i\ge0$:
\begin{align}
\mathbf s_i^{(0)} = \mathbf s_i^{(2)} = \mathbf s_i^{(4)} = \mathbf s_i^{(6)} &= \mathbf b_i^{(0)},\label{eqn-sn-bn-1} \\
\mathbf s_i^{(1)} = \mathbf s_i^{(5)} &= \mathbf b_i^{(1)} - \mathbf b_i^{(0)},\label{eqn-sn-bn-2} \\
\mathbf s_i^{(3)} &= \mathbf b_i^{(2)} - \mathbf b_i^{(0)} - \mathbf b_i^{(1)}. \label{eqn-sn-bn-3}
\end{align}
From Equations \eqref{eq-sn-1} - \eqref{eq-sn-2}, we have
\begin{alignat}{2}
\mathbf s_0^{(0)} = \mathbf s_0^{(2)} = \mathbf s_0^{(4)} = \mathbf s_0^{(6)} &= \mathbf u_0 &&= \mathbf b_0^{(0)},\label{eqn-s0-b0-1} \\
\mathbf s_0^{(1)} = \mathbf s_0^{(5)} &= \mathbf u_1 &&= \mathbf b_0^{(1)} - \mathbf b_0^{(0)}, \label{eqn-s0-b0-2} \\
\mathbf s_0^{(3)} &= \mathbf u_2 &&= \mathbf b_0^{(2)} - \mathbf b_0^{(0)} - \mathbf b_0^{(1)}. \label{eqn-s0-b0-3}
\end{alignat}
Suppose that Equations \eqref{eqn-sn-bn-1} - \eqref{eqn-sn-bn-3} hold for all $i=k$. 
Then from Equations \eqref{eq-sn-3} - \eqref{eq-sn-9}, we have
\begin{alignat}{2}\label{eq-sn1-bn}
\mathbf s_{k+1}^{(0)} = \mathbf s_{k+1}^{(2)} = \mathbf s_{k+1}^{(4)} + \mathbf s_{k+1}^{(6)}
&= \mathbf b_k^{(0)} + \mathbf b_k^{(1)} + \mathbf b_k^{(2)} &&= \mathbf b_{k+1}^{(0)},\nonumber \\
\mathbf s_{k+1}^{(1)} = \mathbf s_{k+1}^{(5)}
&= \mathbf b_k^{(1)} + \mathbf b_k^{(2)} &&= \mathbf b_{k+1}^{(0)} - \mathbf b_{k+1}^{(0)}, \nonumber\\
\mathbf s_{k+1}^{(3)} &= \mathbf b_k^{(2)} &&= \mathbf b_{k+1}^{(2)} - \mathbf b_{k+1}^{(0)} - \mathbf b_{k+1}^{(1)}
\end{alignat}
Therefore, we prove Equations \eqref{eqn-sn-bn-1} - \eqref{eqn-sn-bn-3}.
As a consequence, we have
\begin{align}
\mathbf s_i^{(0)} &= \mathbf b_i^{(0)}, \label{eq-sn-sum-bn-1} \\
\mathbf s_i^{(0)} + \mathbf s_i^{(1)} &= \mathbf b_i^{(1)}, \label{eq-sn-sum-bn-2} \\
\mathbf s_i^{(0)} + \mathbf s_i^{(1)} + \mathbf s_i^{(2)} &= \mathbf b_i^{(0)} + \mathbf b_i^{(1)}, \label{eq-sn-sum-bn-3} \\
\mathbf s_i^{(0)} + \mathbf s_i^{(1)} + \mathbf s_i^{(2)} + \mathbf s_i^{(3)} &= \mathbf b_i^{(2)}, \label{eq-sn-sum-bn-4} \\
\mathbf s_i^{(0)} + \mathbf s_i^{(1)} + \mathbf s_i^{(2)} + \mathbf s_i^{(3)} + \mathbf s_i^{(4)} &= \mathbf b_i^{(0)} + \mathbf b_i^{(2)}, \label{eq-sn-sum-bn-5} \\
\mathbf s_i^{(0)} + \mathbf s_i^{(1)} + \mathbf s_i^{(2)} + \mathbf s_i^{(3)} + \mathbf s_i^{(4)} + \mathbf s_i^{(5)} &= \mathbf b_i^{(1)} + \mathbf b_i^{(2)}, \label{eq-sn-sum-bn-6} \\
\mathbf s_i^{(0)} + \mathbf s_i^{(1)} + \mathbf s_i^{(2)} + \mathbf s_i^{(3)} + \mathbf s_i^{(4)} + \mathbf s_i^{(5)} + \mathbf s_i^{(6)} &= \mathbf b_i^{(0)} + \mathbf b_i^{(1)} + \mathbf b_i^{(2)} = \mathbf b_{i+1}^{(0)}.\label{eq-sn-sum-bn-7}
\end{align}
This proves Equation \eqref{eqn:viwi} and thus Theorem \ref{thm:main-2}.

\section{Self-replicating words and discrete tilings}\label{sec_discrete_tiling_plane} 

Let $[\mathbf w] = [w_0\cdots w_l]$ be a word on $d$-letters.
Let us denote $[\mathbf w_0] = [\mathbf w]$ and for $0\le i \le l$,
\begin{equation}
	[\mathbf w_0^{(i)}] = [w_i].
\end{equation}
For $n\ge 0$, the \textbf{$n$-th replicate} of $[\mathbf w]$,
denoted by  $[\mathbf w_n]$, is the word defined by
\begin{equation}
	[\mathbf w_n] = [\mathbf w_n^{(0)}]\cdots[\mathbf w_n^{(l)}], 
\end{equation}
where for $0\le i \le l$,
\begin{equation}
	[\mathbf w_n^{(i)}] = [\mathbf w_{n-1}^{(0)}]\cdots [\mathbf w_{n-1}^{(l-w_i)}].
\end{equation}

\begin{defn}
Let $\mathbf v_n$ be the word vector for $\mathbf w_n$.
A word $\mathbf w$ is called \textbf{self-replicating} 
if the limit $\mathbf v_\infty = \displaystyle \lim \mathbf v_n/\Vert\mathbf v_n\Vert$ exists.
\end{defn}

We will define the domains for self-replicating words as follows, 
in a similar way to Pisot domains, but using the ideas from the construction B.
Let $P$ the hyperplane orthogonal to $\mathbf v_\infty$. 
Let $\mathbf r_0,\cdots, \mathbf r_{d-2}$ be the orthonormal basis on $P$ obtained by
Gram--Schmidt process on $\mathbf v_\infty, \mathbf e_0, \cdots,\mathbf e_{d-2}$.
Let $\pi:\mathbb R^d\to P$ be the orthogonal projection,
and $\mathbf u_i$, $i=0,\cdots,n-1$, be the $(d-1)$-dimensional vectors  as follows.
\begin{equation}\label{eqn:ui-self}
	\mathbf u_i = (\pi(\mathbf e_i)\circ \mathbf r_0,\cdots,\pi(\mathbf e_i)\circ\mathbf r_{n-2}).
\end{equation}
We define the vectors $\mathbf s_i^{(j)}$ similar to Equations \eqref{eq-sn-1} - \eqref{eq-sn-9}:
for $j=0,\cdots,l$,
\begin{equation}\label{eq-soj-1}
\mathbf s_0^{(j)} = \mathbf u_j,\quad
\mathbf s_{i+1}^{(j)} = \sum_{k=0}^{l-a_j}\mathbf s_i^{(k)}\textrm{ for } i\ge0.
\end{equation}
Finally, using the same definition of the set $W_i$ in Equation \eqref{eqn:wiwi-1},
define the $(d-1)$-dimensional domain $W$ as 
\begin{equation}\label{eqn:w-self}
	W = \bigcup_{i=-1}^\infty W_i.
\end{equation}

\begin{table}
\centering
\begin{tabular}{|c|c|}
\hline
$\mathbf w$ & $\mathbf v_\infty$ \\ \hhline{|=|=|}
0120 & (0.756, 0.521, 0.397) \\\hline
0102 & (0.850, 0.462, 0.251) \\\hline
0201 & (0.831, 0.259, 0.492) \\\hline
0102010 & (0.861, 0.447, 0.242) \\\hline
1201 & (0.381, 0.717, 0.584) \\\hline
2010 & (0.771, 0.359, 0.526) \\\hline
\end{tabular}
\caption{Self replicating words and their limits}
\label{tab:self-repli}
\end{table}

\begin{figure}
    \centering
    \begin{subfigure}[b]{0.3\textwidth}
    \centering
    \includegraphics[scale=.1]{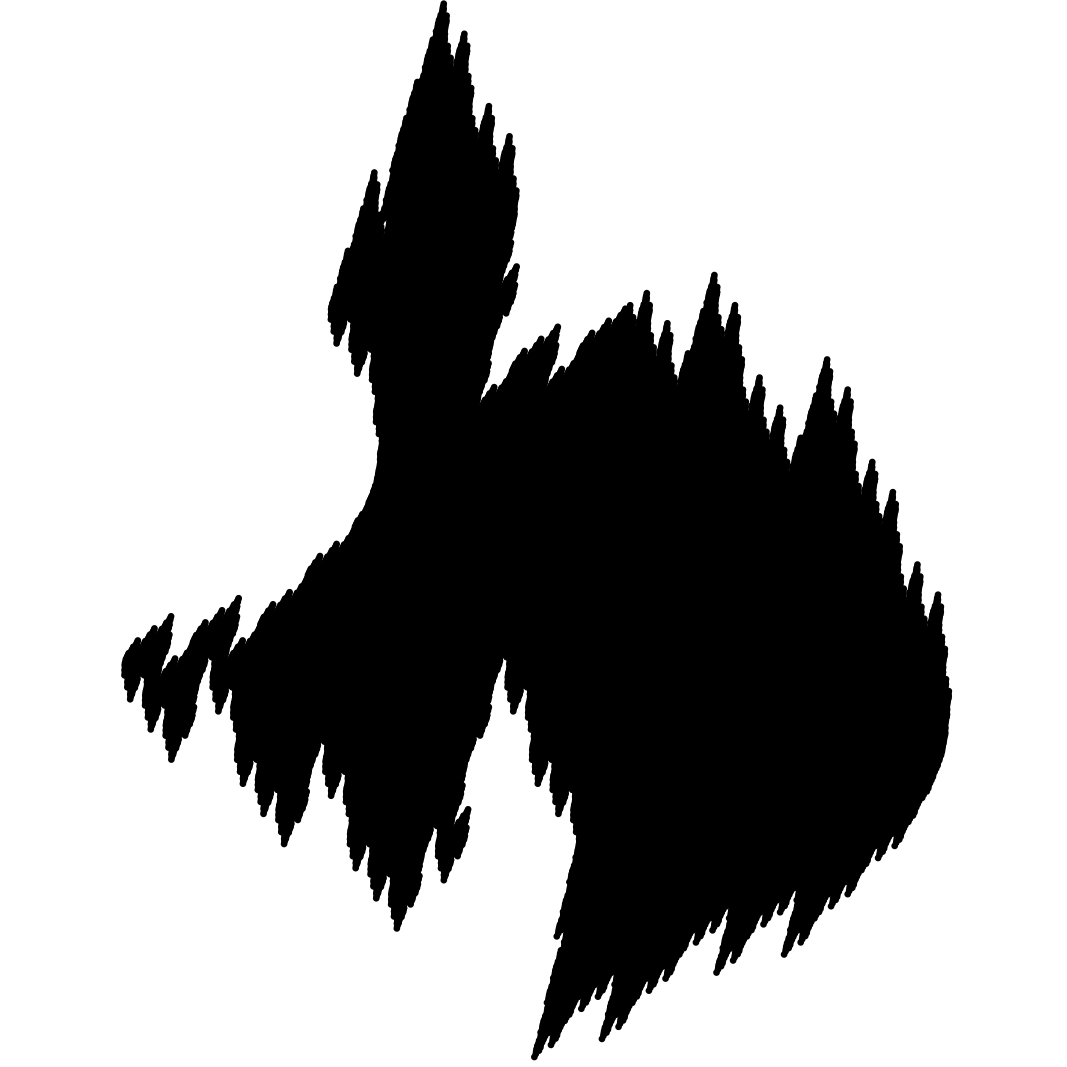}
    \caption{0120 ``The Fish"}
    \label{fig-0120}
    \end{subfigure}
    \begin{subfigure}[b]{0.3\textwidth}
    \centering
    \includegraphics[scale=.1]{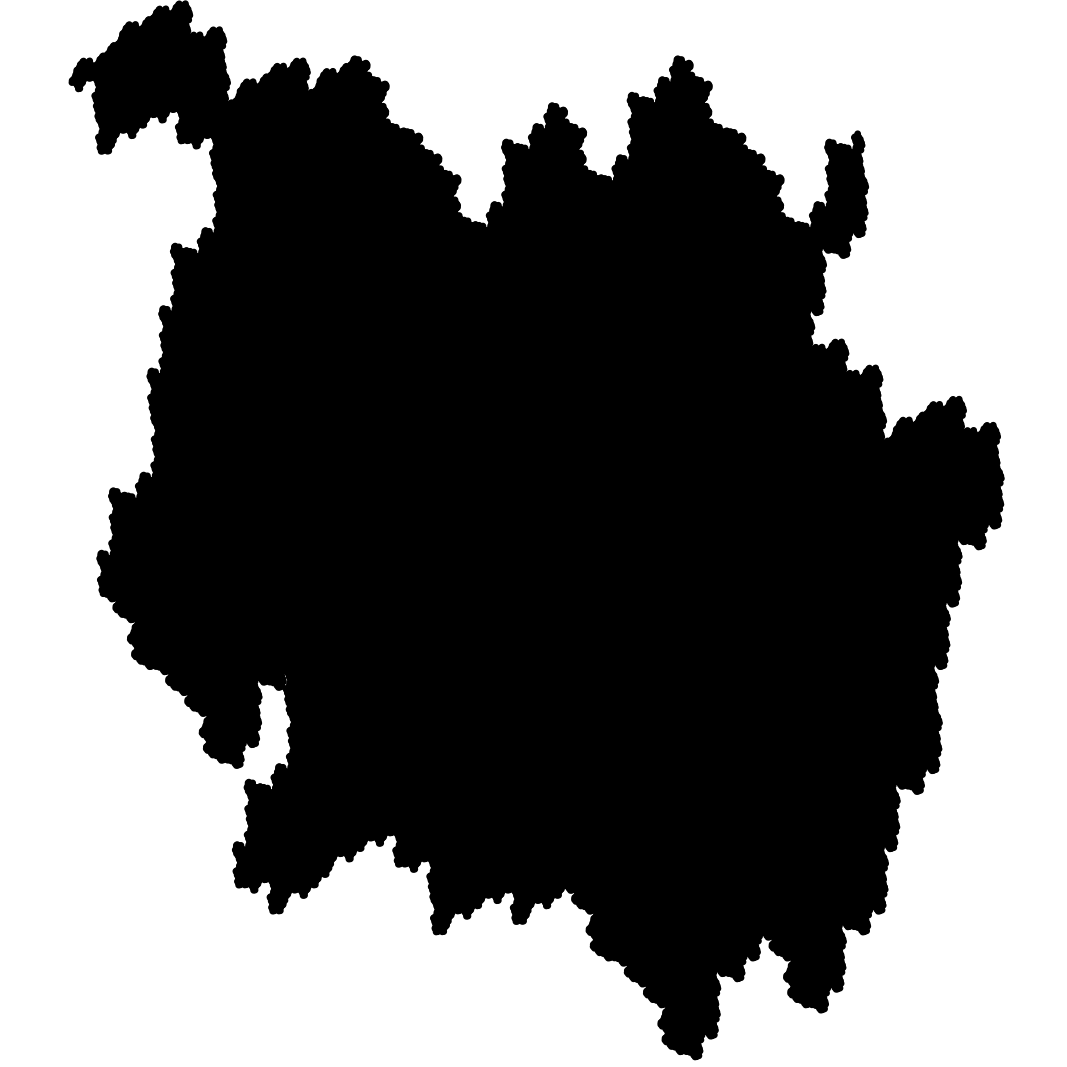}
    \caption{0102}
    \label{fig-0102}
    \end{subfigure}
    \begin{subfigure}[b]{0.3\textwidth}
    \centering
    \includegraphics[scale=.1]{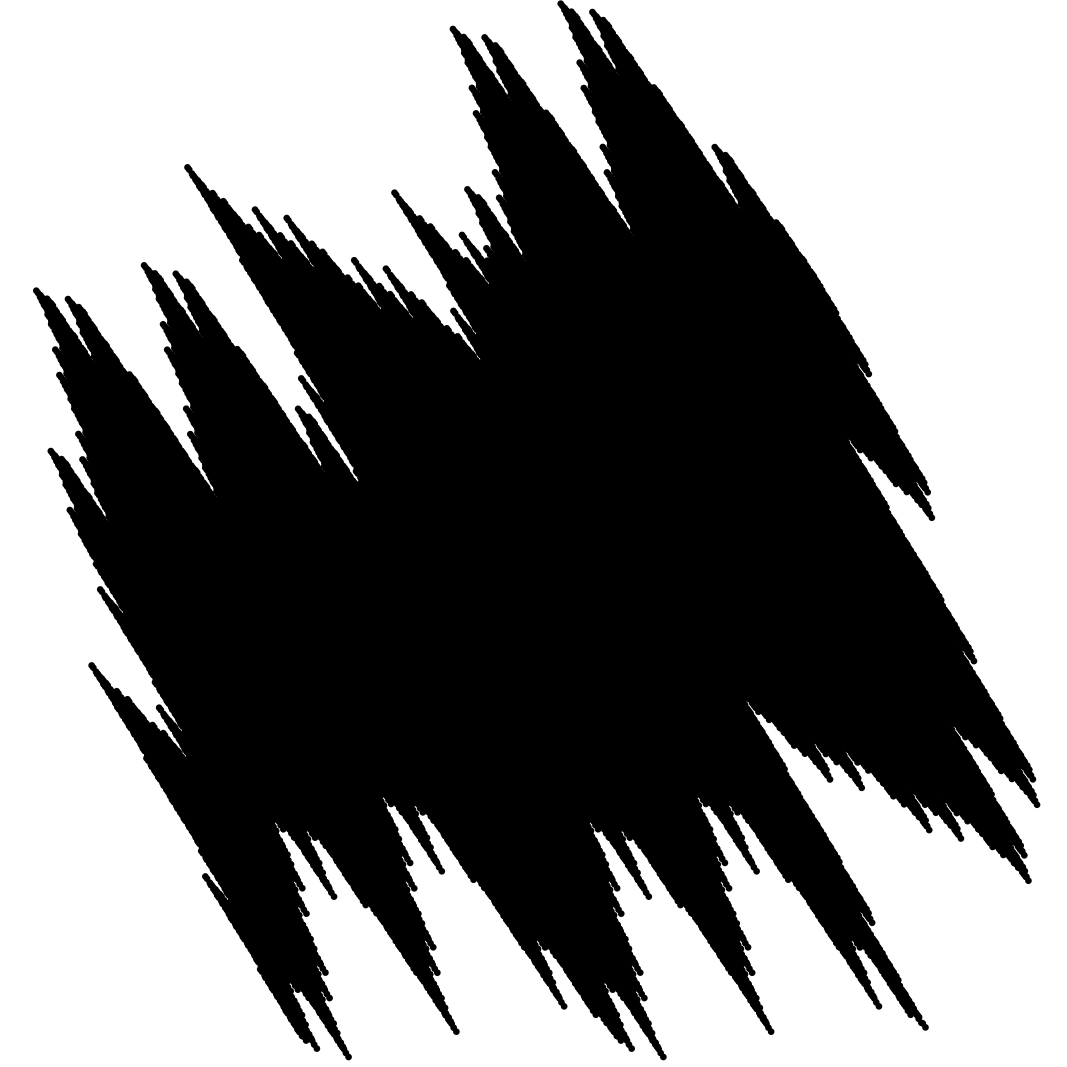}
    \caption{0201}
    \label{fig-0201}
    \end{subfigure}
    \begin{subfigure}[b]{0.3\textwidth}
    \centering
    \includegraphics[scale=.1]{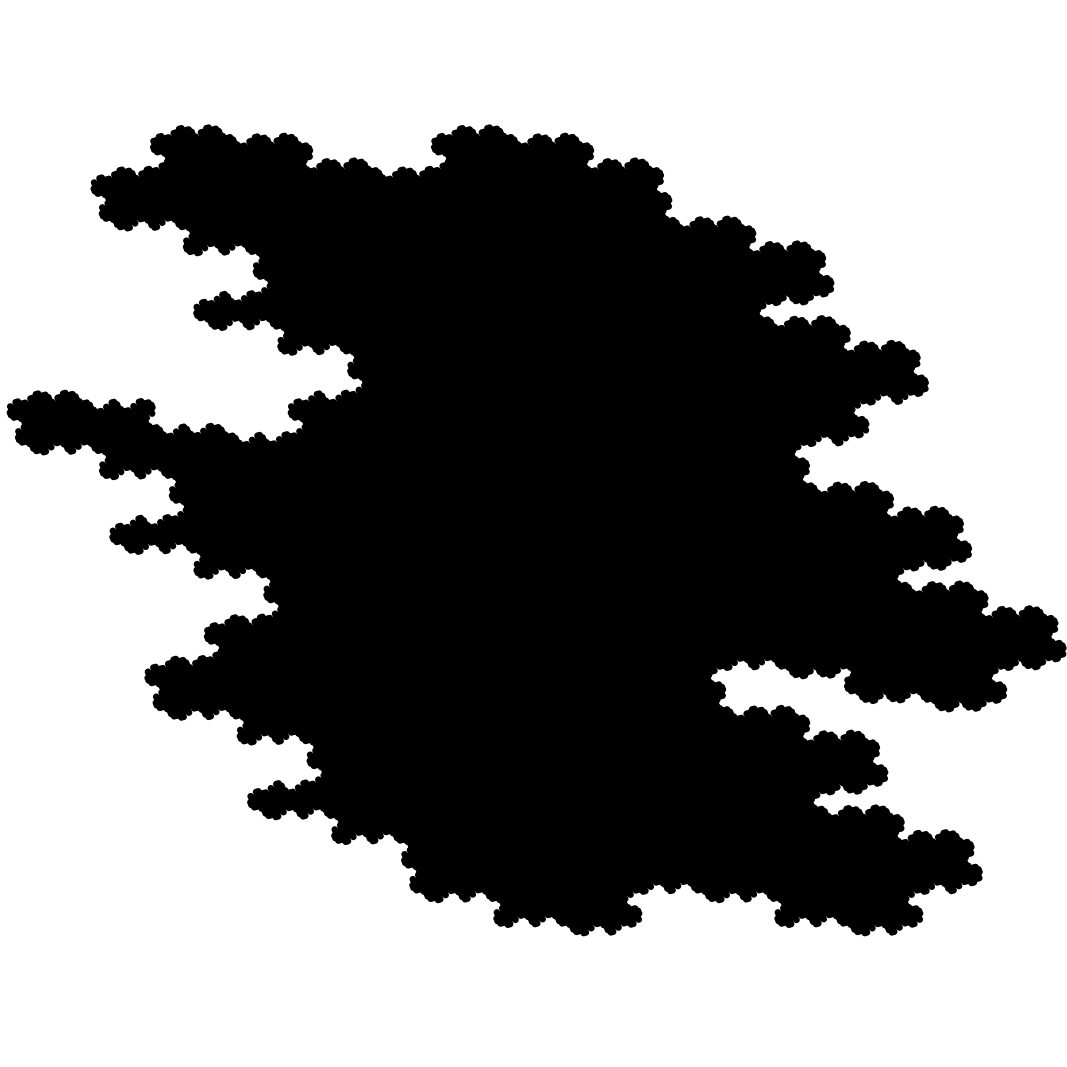}
    \caption{0102010}
    \label{fig-0102010}
    \end{subfigure}
    \begin{subfigure}[b]{0.3\textwidth}
    \centering
    \includegraphics[scale=.1]{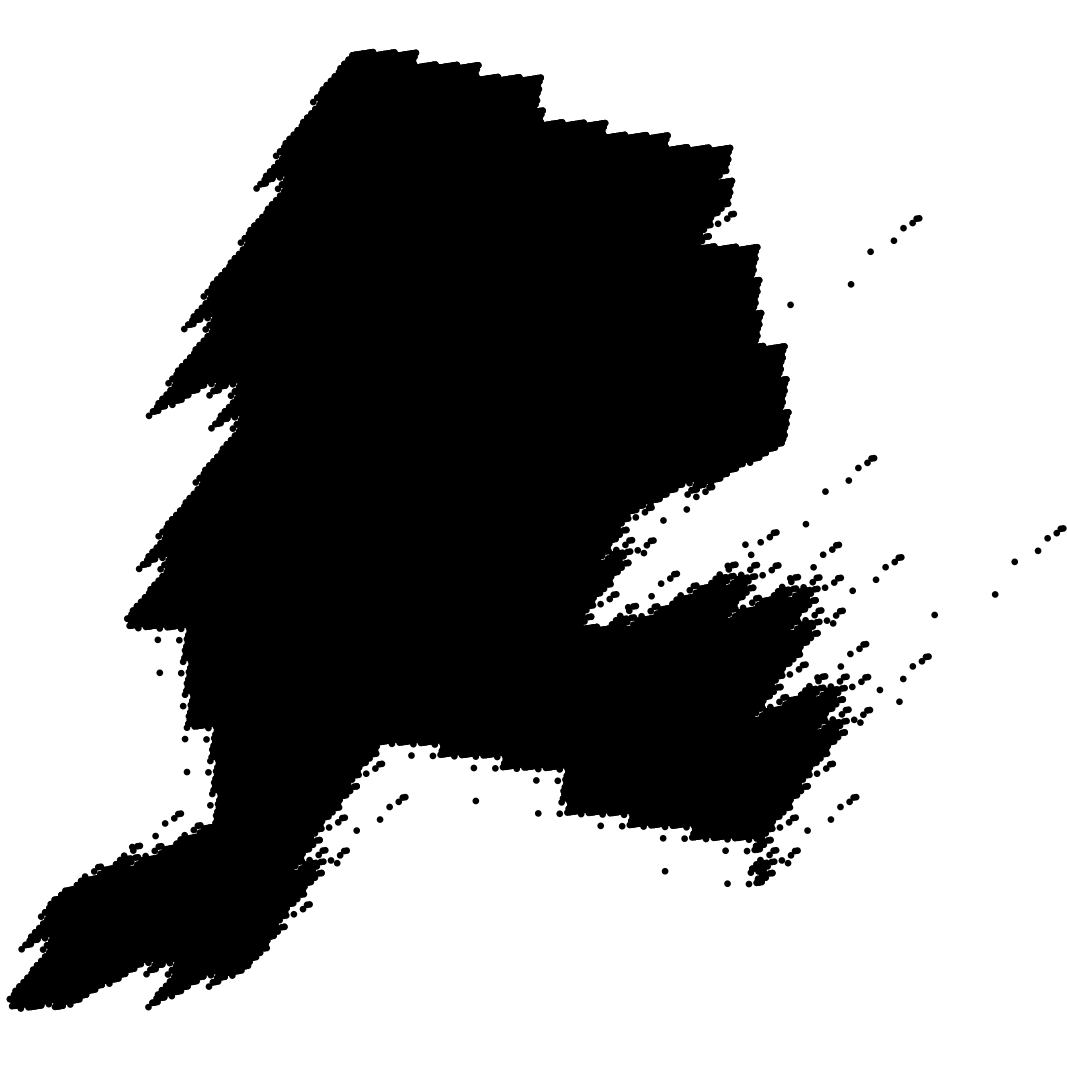}
    \caption{1201 ``The Runner"}
    \label{fig-012}
    \end{subfigure}
    \begin{subfigure}[b]{0.3\textwidth}
    \centering
    \includegraphics[scale=.1]{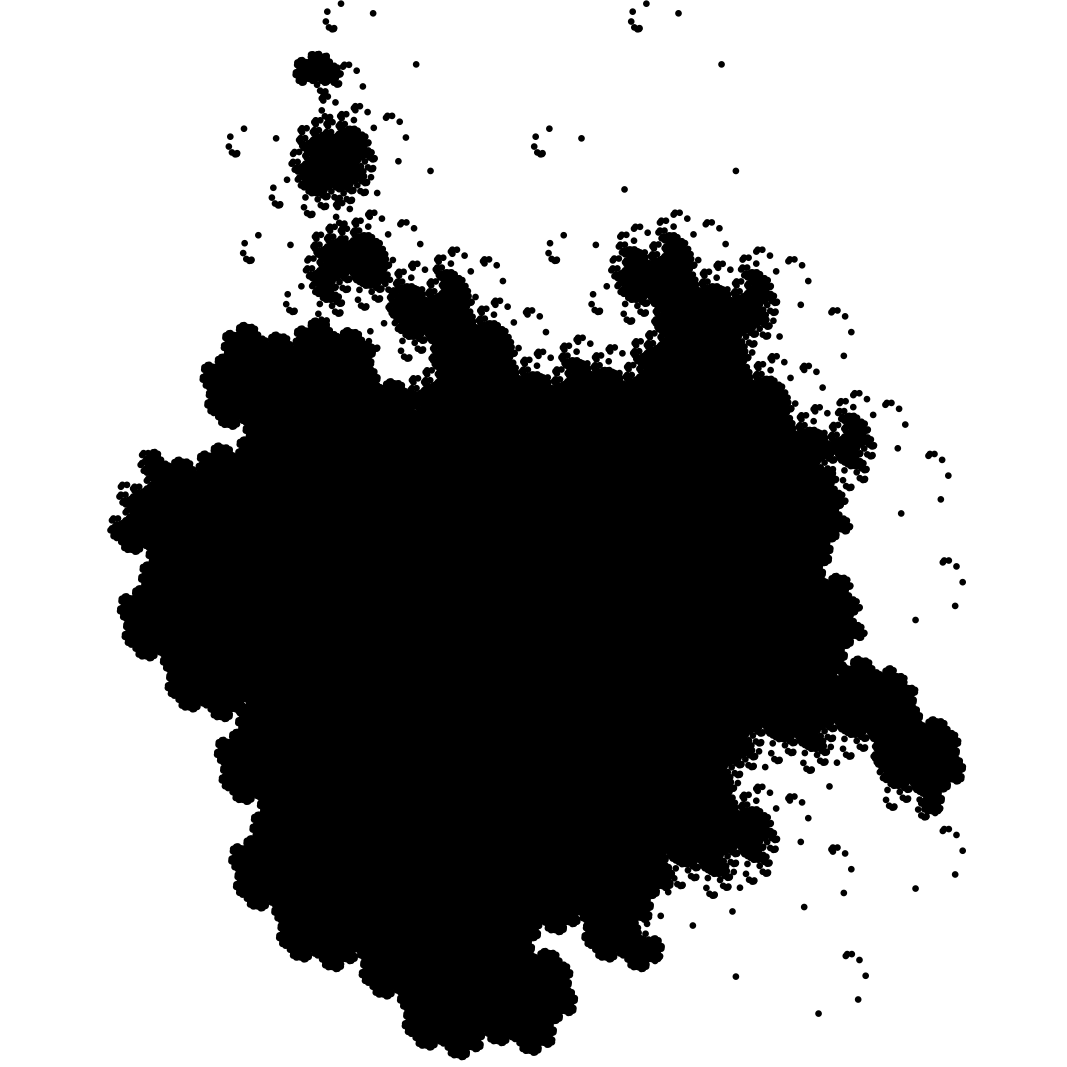}
    \caption{2010}
    \label{fig-2010}
    \end{subfigure}
    \caption{The domains for self-replicating words}
    \label{fig-w-self}
\end{figure}

Table \ref{tab:self-repli} shows some examples of self-replicating words on $3$ letters
and their limits $\mathbf v_\infty$.
Figure \ref{fig-w-self} shows the domain $W$ for each elf-replicating words.
One can see that all domains are compact, and tiles $\mathbb R^2$ discretely.
Once we have $(n-1)$-dimensional orthonormal basis $\mathbf u_i$, $i=0,\cdots,n-2$,
all points in such domain can be obtained by arithmetic summations of vectors as in Equation \eqref{eq-soj-1}.

The domains in Figure \ref{fig-w-self} tiles $\mathbb R^2$.
From $\mathbf u_0, \mathbf u_1, \mathbf u_2$ in Equation \eqref{eqn:ui-self}, the three vectors
\begin{equation}\label{eqn:u01}
\mathbf u_{01} = \mathbf u_0 - \mathbf u_1,\quad \mathbf u_{12} = \mathbf u_1 -\mathbf u_2, \quad \mathbf u_{02} = \mathbf u_0 - \mathbf u_2
\end{equation}
are the vectors that translates the domains into hexagonal tilings.
For example, we can translate ``The Fish" in Figure \ref{fig-0120} by 
linear combinations of these three vectors
$(1.27, -0.48),(-0.34, 1.36), (0.93, 0.88)$.
Figure \ref{fig:0120_tiling} shows the discrete tilings of Figures \ref{fig-0120}, \ref{fig-012}.
At this point, we can conjecture that the following statement is true.

\begin{figure}
\centering
\begin{subfigure}[b]{0.475\textwidth}
\centering
\includegraphics[scale=.15]{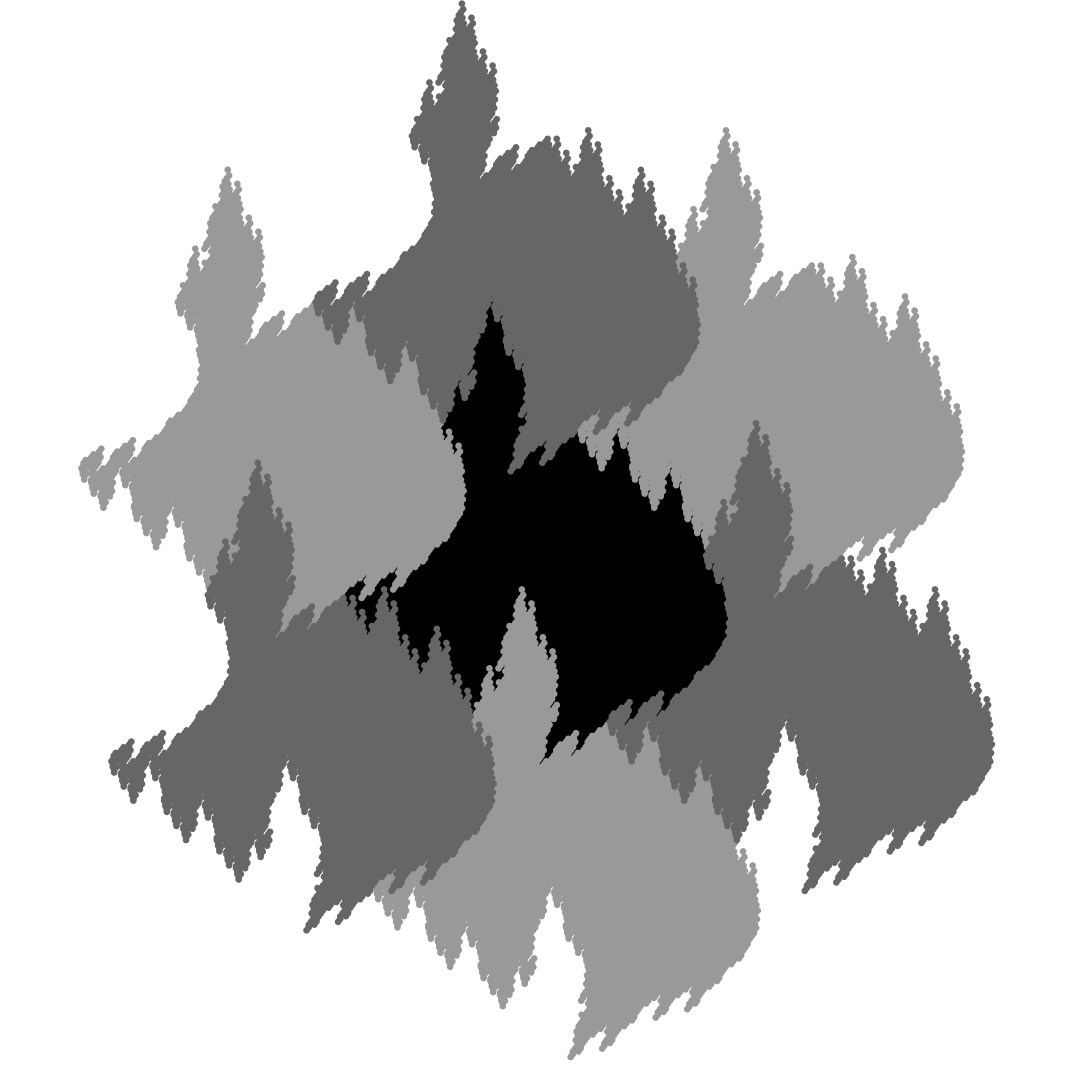}
\caption{The discrete tiling of ``The Fish" }
\label{fig:0120_tiling}
\end{subfigure}
\begin{subfigure}[b]{0.475\textwidth}
\centering
\includegraphics[scale=.15]{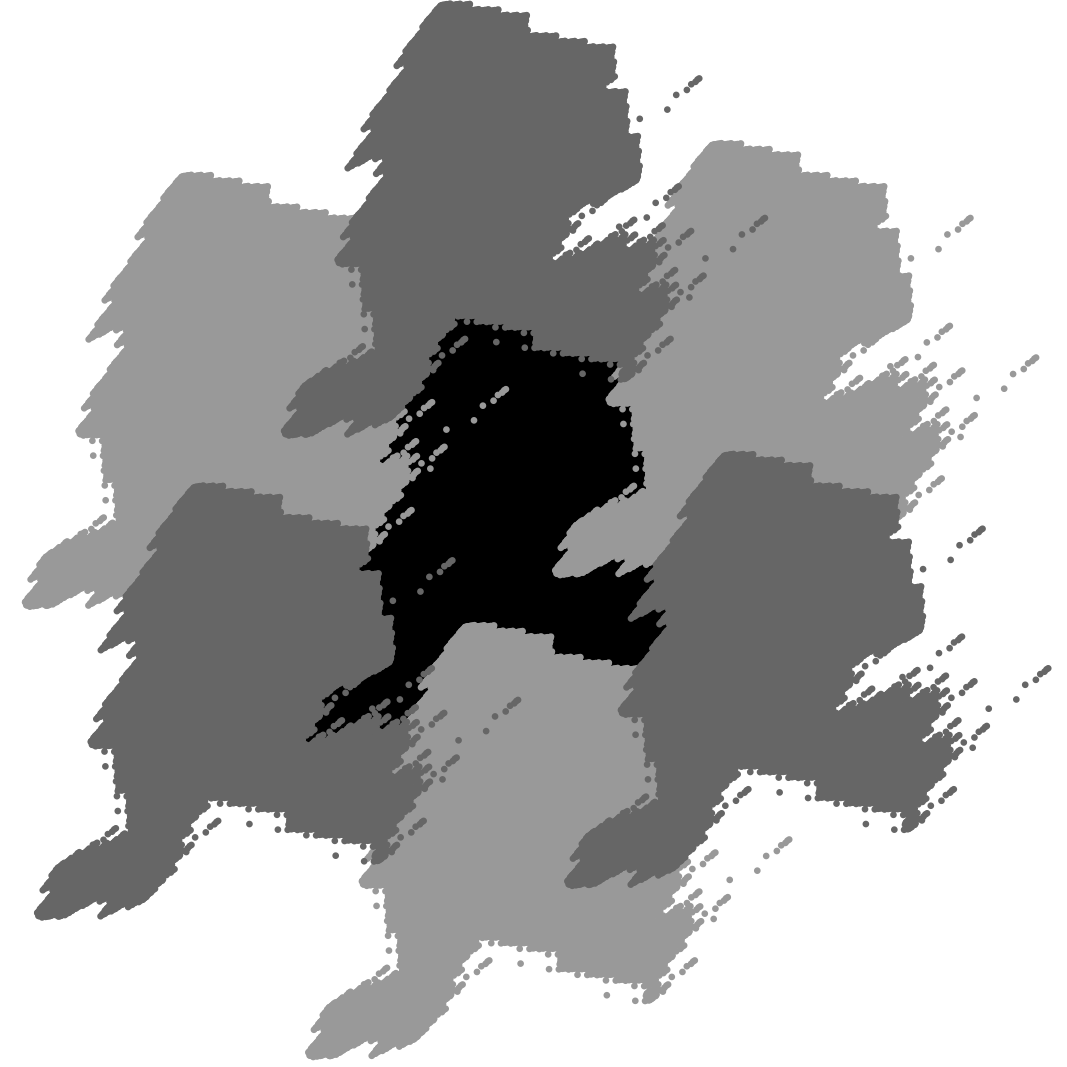}
\caption{The discrete tiling of ``The Runner"}
\label{fig:1201_tiling}
\end{subfigure}
\caption{The discrete tiling of self-replicating domains,}
\label{fig:tiling}
\end{figure}

\begin{conj*}
Let $[\mathbf w]$ be a self-replicating word on three letters. 
Then the domain $W$ defined in Equation \eqref{eqn:w-self} tiles $\mathbb R^2$ discretely
in the following sense: for the vectors $\mathbf u_{01}$, $\mathbf u_{12}$, and $\mathbf u_{02}$
in Equation \eqref{eqn:u01} and for 
$3$-dimensional integral vector $\mathbf c= (c_{01},c_{12}, c_{02})$, define
$W_{\mathbf c} = W + \sum_{i<j}c_{ij}\mathbf u_{ij}$.
Then $W_{\mathbf c_1} \cap W_{\mathbf c_2} = \emptyset$ unless $\mathbf c_1= \mathbf c_2$ and
$\mathbb R^2 = \displaystyle\bigcup_{\mathbf c\in\mathbb Z^3}W_{\mathbf c}$.
\end{conj*}

\section{Conclusion}

The Rauzy fractal attracted mathematical interests for many years, 
and its characteristics have been generalized into many areas of advanced research.
We studied yet another characteristics of the Rauzy fractal with an elementary view point,
and showed that this characteristic can be generalized for creating another 
tiling scheme for the two dimensional Euclidean space.
We expect that our final conjecture can be further generalized to higher dimensions. 

\bibliographystyle{plain}
\bibliography{references}
\end{document}